\documentclass{article}

\PassOptionsToPackage{numbers, compress}{natbib}

\usepackage[final]{neurips_2023}

\usepackage{amsmath,amssymb,amsfonts}

\usepackage[linesnumbered,ruled]{algorithm2e}

\usepackage{enumitem}

\usepackage{mathtools}				
\DeclareMathOperator{\rk}{rank}	
\DeclareMathOperator{\tr}{tr}	    
\DeclareMathOperator{\vecc}{vec}	    
\DeclareMathOperator{\mat}{mat}	    

\DeclarePairedDelimiter\ceil{\lceil}{\rceil}    
\DeclarePairedDelimiter\floor{\lfloor}{\rfloor} 

\newcommand{\RR}{\mathbb R}
\newcommand\mb{\mathbf}

\DeclarePairedDelimiter{\norm}{\lVert}{\rVert}

\usepackage{amsthm}
\theoremstyle{plain}
\newtheorem{theorem}{Theorem}
\newtheorem{proposition}{Proposition}
\newtheorem{corollary}{Corollary}
\newtheorem{lemma}{Lemma}
\theoremstyle{definition}
\newtheorem{definition}{Definition}
\newtheorem*{definition*}{Definition}

\theoremstyle{remark}

\usepackage[utf8]{inputenc} 
\usepackage[T1]{fontenc}    
\usepackage{hyperref}       
\usepackage{url}            
\usepackage{array}
\usepackage{multirow}
\usepackage{booktabs}
\usepackage{amsfonts}       
\usepackage{nicefrac}       
\usepackage{microtype}      
\usepackage{xcolor}         
\usepackage{graphicx}
\usepackage{subcaption}

\title{Algorithmic Regularization in Tensor Optimization: Towards a Lifted Approach in Matrix Sensing}

\author{%
  Ziye Ma\\
  Department of EECS\\
  UC Berkeley\\
  \texttt{ziyema@berkeley.edu} \\
  \And
  Javad Lavaei \\
  Department of IEOR \\
  UC Berkeley\\
  \texttt{lavaei@berkeley.edu} \\
  \And
  Somayeh Sojoudi \\
  Department of EECS, ME \\
  UC Berkeley \\
  \texttt{sojoudi@berkeley.edu} \\
}

\begin{document}

\maketitle

\begin{abstract}
Gradient descent (GD) is crucial for generalization in machine learning models, as it induces implicit regularization, promoting compact representations. In this work, we examine the role of GD in inducing implicit regularization for tensor optimization, particularly within the context of the lifted matrix sensing framework. This framework has been recently proposed to address the non-convex matrix sensing problem by transforming spurious solutions into strict saddles when optimizing over symmetric, rank-1 tensors. We show that, with sufficiently small initialization scale, GD applied to this lifted problem results in approximate rank-1 tensors and critical points with escape directions. Our findings underscore the significance of the tensor parametrization of matrix sensing, in combination with first-order methods, in achieving global optimality in such problems.

\end{abstract}

\section{Introduction}
This paper is dedicated to addressing the non-convex problem of matrix sensing, which has numerous practical applications and is rich in theoretical implications. Its canonical form can be written as:
\begin{align} \label{eq:SDP main}
    \mathrm{find}&\quad M\in\mathbb{R}^{n\times n} \\
    \mathrm{s.t.}& \quad \mathcal{A}(M) = \mathcal{A}(M^*) \quad \rk(M) \leq r, M \succeq 0 \notag .
    \notag
\end{align}
$\mathcal{A}(\cdot): \RR^{n \times n} \mapsto \RR^m$ is a linear operating consisting of $m$ sensing matrices $\{A_i\}_{i=1}^m \in \RR^{n \times n}$ where $\mathcal{A}(M)= [ \langle A_1, M \rangle, \dots, \langle A_m, M \rangle]^T$. The sensing matrices and the measurements $b = \mathcal{A}(M^*)$ are given, while $M^*$ is an unknown low-rank matrix to be recovered from the measurements. The true rank of $M^*$ is bounded by $r$, usually much smaller than the problem size $n$. More importantly, since $\mathcal{A}$ is linear, one can replace $A_i$ with $(A_i+A_i^\top)/2$ without changing $b$, and therefore all sensing matrices can be assumed to be symmetric. 

The aforementioned problem serves as an extension of both compressed sensing \cite{donoho2006compressed}, which is widely applied in the field of medical imaging, and matrix completion \cite{candes2009exact, candes2010power}, which possesses an array of notable applications \cite{nguyen2019low}. Additionally, this problem emerges in a variety of real-world situations such as phase retrieval \cite{singer2011angular, boumal2016nonconvex, shechtman2015phase}, motion detection \cite{fattahi2020exact}, and power system state estimation \cite{zhang2017conic, jin2019towards}. A recent study by \cite{molybog2020conic} established that any polynomial optimization problem can be converted into a series of problems following the structure of \eqref{eq:SDP main}, thereby underscoring the significance of investigating this specific non-convex formulation. Within the realm of contemporary machine learning, \eqref{eq:SDP main} holds relevance as it is equivalent to the training problem for a two-layer neural network with quadratic activations \cite{li2018algorithmic}. In this context, $m$ denotes the number of training samples, $r$ is the size of the hidden layer, and the sensing matrices $A_i=x_ix_i^\top$ are rank-1, with $x_i$ representing the $i^{th}$ datapoint.

To solve \eqref{eq:SDP main}, an increasingly popular approach is the Burer-Monteiro (BM) factorization \cite{burer2003nonlinear}, in which the low-rank matrix $M$ is factorized into $M = XX^\top$ with $X \in \RR^{n \times r}$, thereby omitting the constraint, making it amenable to simple first-order methods such as gradient descent (GD), while scaling with $\mathcal{O}(nr)$ instead of $\mathcal{O}(n^2)$. The formulation can be formally stated as follows:
\begin{equation}\label{eq:unlifted_main}
	\min_{X \in \mathbb{R}^{n \times r}} f(X) \coloneqq \frac{1}{2} \|\mathcal{A}(XX^T) - b\|^2 = \frac{1}{2} \|\mathcal{A}(XX^T-ZZ^\top)\|^2 \quad \text{(Unlifted Problem)}
\end{equation}
with $Z \in \RR^{n \times r}$ being any ground truth representation such that $M^* = ZZ^\top$. Since \eqref{eq:unlifted_main} is a non-convex problem, it can have spurious local minima\footnote{A spurious point satisfies first-order and second-order necessary conditions but is not a global minimum.}, making it difficult to recover $M^*$ in general. The pivotal concept in solving \eqref{eq:SDP main} and \eqref{eq:unlifted_main} to optimality is the notion of Restricted Isometry Property (RIP), which measures the proximity between $\|\mathcal{A}(M)\|^2_F$ and $\|M\|^2_F$ for all low-rank matrices $M$. This proximity is captured by a constant $\delta_p$, where $\delta_{p} = 0$ means $\mathcal{A}(M)=M$ for matrices up to rank $p$, leading to exact isometry case, and $\delta_p \rightarrow 1$ implying a problematic scenario in which the proximity error is large. For a precise definition,  please refer to Appendix \ref{sec:app_def}. 

Conventional wisdom suggests that there is a sharp bound on the RIP constant that controls the recoverability of $M^*$, with $1/2$ being the bound for \eqref{eq:unlifted_main}. \cite{zhang2019sharp,ma2022sharp} prove that if $\delta_{2r} < 1/2$, then all local minimizers are global minimizers, and conversely if $\delta_{2r} \geq 1/2$, counterexamples can be easily established. Similar bounds of $1/3$ are also derived for general objectives \cite{ha2020equivalence,zhang2021general}, demonstrating the importance of the notion of RIP. However, more recent studies reveal that the technique of over-parametrization (by using $X \in \RR^{n \times r_{\text{search}}}$ instead, with $r_{\text{search}} > r$) can take the sharp RIP bound to higher values \cite{zhang2021sharp,zhang2022improved}. Recently, it has also been shown that using a semidefinite programming (SDP) formulation (convex relaxation) can lead to guaranteed recovery with a larger RIP bound that approaches 1 in the transition to the high-rank regime when $n \approx 2r$ \cite{yalcin2022semidefinite}. These works all show the efficacy of over-parametrization, shedding light on a powerful way to find the global solution of complex non-convex problems. However, all of these techniques fail to handle real-world cases with $\delta_{2r} \rightarrow 1$ in the low-rank regime. To this end, a recent work \cite{ma2023over} drew on important concepts from the celebrated Lasserre’s Hierarchy \cite{lasserre2001global} and proposed a lifted framework based on tensor optimization that could convert spurious local minimizers of \eqref{eq:unlifted_main} into strict saddle points in the lifted space, for arbitrary RIP constants in the $r=1$ case. We state this lifted problem below:
\begin{equation}\label{eq:lifted_rank1}
	\min_{\mb{w} \in \RR^{n \circ l}} \quad \| \langle \mb{A}^{\otimes l}, \mb{w} \otimes \mb{w} \rangle - b^{\otimes l} \|^2_F \quad \text{(Lifted Problem, $r=1$)}
\end{equation}
where $\mb{w}$ is an $l$-way, $n$-dimensional tensor, and $\mb{A}^{\otimes l}$ and $b^{\otimes l}$ are tensors "lifted" from $\mathcal{A}$ and $b$ via tensor outer product. We defer the precise definition of tensors and their products to Section \ref{sec:def}. The main theorem of \cite{ma2023over} states that when $r=1$, for some appropriate $l$, the first-order points (FOP) of \eqref{eq:unlifted_main} will be converted to FOPs of \eqref{eq:lifted_rank1} via lifting, and that spurious second order points (SOP) of \eqref{eq:unlifted_main} will be converted into strict saddles, under some technical conditions, provided that $\mb{w}$ is symmetric, and rank-1. This rank-1 constraint on the decision variable $\mb{w}$ is non-trivial, since finding the dominant rank-1 component of symmetric tensors is itself a non-convex problem in general, and requires a number of assumptions for it to be provably correct \cite{kofidis2002best,wu2020symmetric}. This does not even account for the difficulties of maintaining the symmetric properties of tensors, which also has no natural guarantees. Therefore, although this lifted formulation may be promising in the pursuit of global minimum, there are still major questions to be answered. Most importantly, it is desirable to know \emph{whether the symmetric, rank-1 condition is necessary, and if so, how to achieve it without explicit constraints?}

The necessity of the condition in question can be better understood through insights from \cite{levin2022effect}. The authors argue that over-parametrizing non-convex optimization problems can reshape the optimization landscape, with the effect being largely independent of the cost function and primarily determined by the parametrization. This notion is consistent with \cite{ma2023over}, which contends that over-parametrizing vectors into tensors can transform spurious local solutions into strict saddles. However, \cite{levin2022effect} specifically examines the parametrization from vectors/matrices to tensors, concluding that stationary points are not generally preserved under tensor parametrization, contradicting \cite{ma2023over}. This implies that the symmetric, rank-1 constraints required in \eqref{eq:lifted_rank1} are crucial for the conversion of spurious points. 



It is essential to devise a method to encourage tensors to be near rank-1, with implicit regularization as a potential solution. There has been a recent surge in examining the implicit regularization effects in first-order optimization methods, such as gradient descent (GD) and stochastic gradient descent (SGD) \cite{li2022survey}, which has been well-studied in matrix sensing settings \cite{stoger2021small,jin2023understanding,ma2022global,li2018algorithmic}. This intriguing observation has prompted us to explore the possible presence of similar implicit regularization in tensor spaces. Our findings indicate that when applying GD to the tensor optimization problem \eqref{eq:lifted_rank1}, an implicit bias can be detected with sufficiently small initialization points. This finding does not directly extend from its matrix counterparts due to the intricate structures of tensors, resulting in a scarcity of useful identities and well-defined concepts for even fundamental properties such as eigenvalues. Furthermore, we show that when initialized at a symmetric tensor, the entire GD trajectory remains symmetric, completing the requirements.

In this paper, we demonstrate that over-parametrization alone does not inherently simplify non-convex problems. However, employing a suitable optimization algorithm offers a remarkably straightforward solution, as this specific algorithm implicitly constrains our search to occur within a compact representation of the over-parametrized space without necessitating manual embeddings or transformations. This insight further encourages the investigation of a (parametrization, algorithm) pair for solving non-convex problems, thereby enhancing our understanding of achieving global optimality in non-convex problems.

\subsection{Related Works}\label{sec:related}
\textbf{Over-parametrization in matrix sensing}. Except for the lifting formulation \eqref{eq:lifted_rank1}, there are two mainstream approaches to over-parametrization in matrix sensing. The first one is done via searching over $Y \in \RR^{n \times r_{\text{search}}}$ instead of $X \in \RR^{n \times r}$, and using some distance metric to minimize the distance between $\mathcal{A}(YY^\top)$ and $b$. Using an $l_2$ norm, \cite{zhang2021sharp,zhang2022improved} established that if $r_{\text{search}} > r[(1+\delta_n)/(1-\delta_n)-1]^2/4$, with $r \leq r_{\text{search}} < n$, then every second-order point $\hat Y \in \RR^{n \times r_{\text{search}}}$ satisfies that $\hat Y \hat Y^\top = M^*$. \cite{ma2023geometric} showed that even in the over-parametrized regime, noise can only finitely influence the optimization landscape. \cite{ma2022global} offered similar results for an $l_1$ loss under good enough RIP constant. Another popular approach to over-paramerization is to use a convex SDP formulation, which is a convex relaxation of \eqref{eq:SDP main} \cite{recht2010guaranteed}. It has been known for years that as long as $\delta_{2r} < 1/2$, then the global optimality of the SDP formulation correspond to the ground truth $M^*$ \cite{cai2013sharp}. Recently \cite{yalcin2022semidefinite} updated this bound to $2r/(n+(n-2r)(2l-5))$, which can approach $1$ if $n \approx 2r$.

\textbf{Algorithm regularization in over-parametrized matrix sensing}. \cite{li2018algorithmic,zhuo2021computational} prove that the convergence to global solution via GD is agnostic of $r_{\text{search}}$, in that it only depends on initialization scale, step-size, and RIP property. \cite{ma2022global} demonstrates the same effect for an $l_1$ norm, and further showed that a small initialization nullifies the effect of over-parametrization. Besides these works, \cite{stoger2021small} refined this analysis, showing that via a sufficiently small initialization, the GD trajectory will make the solution implicitly penalize towards rank-$r$ matrices after a small number of steps. \cite{jin2023understanding} took it even further by showing that the GD trajectory will first make the matrix rank-$1$, rank-$2$, all the way to rank-$r$, in a sequential way, thereby resembling incremental learning. 

\textbf{Implicit bias in tensor learning}. The line of work \cite{razin2021implicit,razin2022implicit,ge2021understanding} demonstrates that for a class of tensor factorization problems, as long as the initialization scale is small, the learned tensor via GD will be approximately rank-1 after an appropriate number of steps. Our paper differs from this line of work in three meaningful ways: 1) The problem considered in those works are optimization problems over vectors, not tensors, and therefore the goal is to learn the structure of a known tensor, rather than learning a tensor itself; 2) Our proof relies directly on tensor algebra instead of adopting a dynamical systems perspective, providing deeper insights into tensor training dynamics while dispensing with the impractical assumption of an infinitesimal step-size.

\subsection{Main Contributions}
\begin{enumerate}
\item We demonstrate that, beyond vector and matrix learning problems, optimization of differentiable objectives, such as the $l_2$ norm, through Gradient Descent (GD) can encourage a more compact representation for tensors as decision variables. This results in tensors being approximately rank-1 after a number of gradient steps. To achieve this, we employ an innovative proof technique grounded in tensor algebra and introduce a novel tensor eigenvalue concept, the variational eigenvalue (v-eigenvalue), which may hold independent significance due to its ease of use in optimization contexts.
\item We show that if a tensor is a first-order point of the lifted objective \eqref{eq:lifted_rank1} and is approximately rank-1, then its rank-1 component can be mapped to an FOP of \eqref{eq:unlifted_main}, implying that all FOPs of \eqref{eq:lifted_rank1} lie in a small sphere around the lifted FOPs of \eqref{eq:unlifted_main}. Furthermore, these FOPs possess an escape direction when reasonably distant from the ground truth solution, irrespective of the Restricted Isometry Property (RIP) constants.
\item We present a novel lifted framework that optimizes over symmetric tensors to accommodate the over-parametrization of matrix sensing problems with arbitrary $r$. This approach is necessary because directly extending the work of \cite{ma2023over} from $r=1$ to higher values  may lead to non-cubical and, consequently, non-symmetric tensors.
\end{enumerate}

\section{Preliminaries} \label{sec:def}
Please refer to Appendix \ref{sec:app_notations} and \ref{sec:app_fopsop} for the notations and definitions of first-order and second-order conditions. Here, we introduce two concepts that are critical in understanding our main results.

\begin{definition}[Tensors and Products]
	We define an $l$-way tensor as:
	\[
		\mathbf{a} = \{a_{i_1i_2\dots i_l} | 1 \leq i_k \leq n_k, 1 \leq k \leq l \} \in \RR^{n_1 \times \dots \times n_l}
	\]
	Moreover, if $n_1 = \dots = n_l$, then we call this tensor an $l$-order (or $l$-way), $n$-dimensional tensor. $\RR^{n \circ l}$ is an abbreviated notion for $n \circ l \coloneqq n \times \dots \times n$. In this work, tensors are denoted with bold letters unless specified otherwise. The tensor outer product, denoted as $\otimes$, of 2 tensors $\mathbf{a}$ and $\mathbf{b}$, respectively of orders $l$ and $p$, is a tensor of order $l+p$, namely $\mb{c} = \mb{a} \otimes \mb{b}$ with $c_{i_1\dots i_l j_1 \dots j_p} = a_{i_1\dots i_l} b_{j_1 \dots j_p}$. 
	We also use the shorthand $\mb{a}^{\otimes l}$ for repeated outer product of $l$ times for arbitrary tensor/matrix/vector $\mb{a}$. $\langle \mb{a}, \mb{b} \rangle_{i_1,\dots,i_d}$ denotes tensor inner product along dimensions $i_1,\dots,i_d$ (with respect to the first tensor), in which we simply sum over the specified dimensions after the outer product $\mb{a} \otimes \mb{b}$ is calculated. This means that the inner product is of $l+p-2d$ orders. Please refer to Appendix \ref{sec:app_def} for a more in-depth review on tensors, especially on its symmetry and rank.
\end{definition}

\begin{definition}[Restricted Strong Smoothness (RSS) and Restricted Strong Convexity (RSC)]
	The linear operator $\mathcal{A}: \RR^{n \times n} \mapsto \RR^m$ satisfies the $(L_s,r)$-RSS property and the $(\alpha_s,r)$-RSC property if
	\begin{align*}
		&f(M) - f(N) \leq \langle M - N, \nabla f(N) \rangle + \frac{L_s}{2} \|M-N\|^2_F \\
		&f(M) - f(N) \geq \langle M - N, \nabla f(N) \rangle + \frac{\alpha_s}{2} \|M-N\|^2_F
	\end{align*}
	are satisfied, respectively for all $M,N \in \mathbb{R}^n$ with $\rk(M), \rk(N) \leq r$. Note that RSS and RSC provide a more expressible way to represent the RIP property, with $\delta_{r} = (L_s-\alpha_s)/(L_s+\alpha_s)$.
\end{definition}

\section{The Lifted Formulation for General $r$}
A natural extension of \eqref{eq:lifted_rank1} to general $r$ requires that instead of optimizing over $X \in \RR^{n \times r}$, we optimize over $\RR^{[n \times r] \circ l}$ tensors, and simply making tensor outer products between $\mb{w}$ to be inner products. However, such a tensor space is non-cubical, and subsequently not symmetric. This is the higher-dimensional analogy of non-square matrices, which lacks a number of desirable properties, as per the matrix scenario. In particular, it is necessary for our approach to optimize over a cubical, symmetric tensor space since in the next section we prove that there exists an implicit bias of the gradient descent algorithm under that setting.

In order to do so, we simply vectorize $X \in \RR^{n \times r}$ into $\vecc(X) \in \RR^{nr}$, and optimize over the tensor space of $\RR^{nr \circ l}$, which again is a cubical space. In order to convert a tensor $\mb{w} \in \RR^{nr \circ l}$ back to $\RR^{[n \times r] \circ l}$ to use a meaningful objective, we introduce a new 3-way permutation tensor $\mb{P} \in \RR^{n \times r \times nr}$ that "unstacks" vectorized matrices. Specifically,
\[
	\langle \mb{P}, \vecc(X) \rangle_3 = X \quad \forall X \in \RR^{n \times r}, n, r \in \mathbb{Z^+}
\]
Such $\mb{P}$ can be easily constructed via filling appropriate scalar "1"s in the tensor. Via Lemma~\ref{lem:kron_iden}, we also know that 
\begin{equation}\label{eq:highr_vectorization}
	 \langle \mb{P}^{\otimes l}, \vecc(X)^{\otimes l} \rangle_{3*[l]}= (\langle \mb{P}, \vecc(X) \rangle_3)^{\otimes l}= X^{\otimes l}
\end{equation}
where $[l]$ denotes the integer set $[1,\dots,l]$, and $c*[l]$ denotes $[c,2c,\dots,c*l]$ for some $c \in \mathbb{Z}^+$. For notational convenience, we abbreviate $\langle \mb{P}^{\otimes l}, \mb{w} \rangle_{3*[l]}$ as $\mb{P}(\mb{w})$ for any arbitrary $z$-dimensional tensor $\mb{w}$ where $z$ can be broken down into the product of two positive integers. Thus, using \eqref{eq:highr_vectorization}, we can extend \eqref{eq:lifted_rank1} to a problem of general $r$, yet still defined over a cubical tensor space:
\begin{equation}\label{eq:lifted_main_highr}
	\min_{\mb{w} \in \RR^{nr \circ l}} \quad \| \langle \mb{A}^{\otimes l}, \langle \mb{P}(\mb{w}), \mb{P}(\mb{w}) \rangle_{2*[l]} \rangle - b^{\otimes l} \|^2_F \quad \text{(Lifted formulation, general $r$)}
\end{equation}
Let us define a 3-way tensor $\mathbf{A} \in \RR^{m \times n \times n}$ so that $\mathbf{A}_{kij} = (A_k)_{ij} \ \forall k \in [m], (i,j) \in [n] \times [n]$. Define $f^l(\cdot): \RR^{n \circ 2l} \mapsto \RR$ and $h^l(\cdot): \RR^{[n \times r] \circ l} \mapsto \RR$ as $f^l(\mb{M}) \coloneqq \| \langle \mb{A}^{\otimes l}, \mb{M} \rangle - b^{\otimes l} \|^2_F$ and $h^l(\mb{w}) = f^l(\langle \mb{w}, \mb{w} \rangle_{2*[l]})$, with $\nabla f^l(\cdot) = \nabla_\mb{M} f^l(\cdot)$ and $\nabla h^l(\cdot) = \nabla_\mb{w} h^l(\cdot)$.

We prove that \eqref{eq:lifted_main_highr} has all the good properties detailed in \cite{ma2023over} for \eqref{eq:lifted_rank1}. In particular, we prove that the symmetric, rank-1 FOPs of \eqref{eq:lifted_main_highr} have a one-to-one correspondence with those of \eqref{eq:unlifted_main}, and that those FOPs that are reasonably separated from $M^*$ or have a small $r^{th}$ singular value can be converted to strict saddle points via some level of lifting. For the detailed theorems and proofs, please refer to Appendix \ref{sec:app_highr}.

\section{Implicit Bias of Gradient Descent in Tensor Space}

In this section, we study why and how applying gradient descent to \eqref{eq:lifted_main_highr} will result in an implicit bias towards to rank-1 tensors. Prior to presenting the proofs, we shall elucidate the primary intuition behind how GD contributes to the implicit regularization of \eqref{eq:unlifted_main}. This will aid in comprehending the impact of implicit bias on \eqref{eq:lifted_main_highr}, as they share several crucial observations, albeit encountering greater technical hurdles. Consider the first gradient step of \eqref{eq:unlifted_main}, initialized at a random point $X_0 \in \RR^{n \times r_{\text{search}}} = \epsilon X$ with $\|X\|^2_F = 1$ and $r_{\text{search}} \geq r$:
 \begin{align*}
 	X_1 &= X_0 - \eta \nabla h(X_0) = \left( I + \eta \left[ \mathcal{A}^* \mathcal{A}(M^*) \right] \right) X_0 - \left[ \mathcal{A}^* \mathcal{A}(X_0 X_0^\top) \right] X_0 \\
 	&= \left( I + \eta \left[ \mathcal{A}^* \mathcal{A}(M^*) \right] \right) X_0 - \epsilon^2 \left[ \mathcal{A}^* \mathcal{A}(X X^\top) \right] X_0 \\
 	& =  \left( I + \eta \left[ \mathcal{A}^* \mathcal{A}(M^*) \right] \right) X_0 + \mathcal{O}(\epsilon^3)
 \end{align*}
 where $\eta$ is the step-size. Therefore, if $\epsilon$ is chosen to be small enough, we have that
 \[
 	X_t \approx (I + \eta \mathcal{A}^* \mathcal{A}(M^*))^t X_0 \quad \text{as} \ \epsilon \rightarrow 0
 \]
 Again, according to the symmetric assumptions on $\mathcal{A}$, we can apply spectral theorem on $\mathcal{A}^* \mathcal{A}(M^*) = \sum_{i=1}^n \lambda_i v_i v_i^\top$ for which the eigenvectors are orthogonal to each other. It follows that $X_t \approx \left( \sum_{i=1}^n (1+\eta\lambda_i)^t v_i v_i^\top \right) X_0$.
 
 In many papers surveyed above on making an argument of implicit bias, it is assumed that there is very strong geometric uniformity, or under the context of this paper, it means that $L_s/\alpha_s \approx 1$. Under this assumption, we have $f(M) \approx f(N) + \langle M-N, \nabla f(M) \rangle +  \|M-N\|^2_F/2$, leading to the fact that $\nabla^2 f(M) = \mathcal{A}^* \mathcal{A} \approx I$. This immediately gives us $\mathcal{A}^* \mathcal{A}(M^*) \approx M^*$ so that $\lambda_{r+1}, \dots, \lambda_n \approx 0$ as $M^*$ is by assumption a rank-$r$ matrix. This further implies that $X_t \approx \left( \sum_{i=1}^r (1+\eta\lambda_i)^t v_i v_i^\top \right) X_0$, which will become a rank-r matrix, achieving the effect of implicit regularization, as $X$ is now over-parametrized by having $r_{\text{search}} \geq r$. 
 
 However, when tackling the implicit regularization problem in tensor space, one key deviation from the aforementioned procedure is that $L_s/\alpha_s$ will be relatively large, as otherwise there will be no spurious solutions, even in the noisy case \cite{zhang2019sharp,ma2023geometric}, which is also the motivation for using a lifted framework in the first place. Therefore, instead of saying that $\mathcal{A}^* \mathcal{A}(M^*) \approx M^*$, we aim to show that the gap between the eigenvalues of a comparable tensor term will enlarge as we increase $l$, making the tensor predominantly rank-1. This observation demonstrates the power of the lifting technique, while at the same time eliminates the critical dependence on a small $L_s/\alpha_s$ factor that is in practice often unachievable due to requiring sample numbers $m$ in the asymptotic regime \cite{candes2011tight}.
 
Therefore, in order to establish an implicit regularization result for \eqref{eq:lifted_main_highr}, there are four major steps that need to be taken:
\begin{enumerate}
	\item Proving that a point on the GD trajectory $\mb{w}_t$ admits a certain breakdown in the form $\mb{w}_t = \langle \mb{Z}_t, \mb{w}_0 \rangle - \mb{E}_t$ for some $\mb{Z}_t$ and $\mb{E}_t$.
	\item Proving that the spectral norm (equivalence of largest singular value) of $\mb{E}_t$ is small (scales with initialization scale $\epsilon$)
	\item Proving that $\langle \mb{Z}_t, \mb{w}_0 \rangle$ has a large separation between its largest and second largest eigenvalues using a tensor version of Weyl's inequality.
	\item Showing that, with the above holding true, $\mb{w}_t$ is predominantly rank-1 after some step $t_*$.
\end{enumerate}
Lemmas \ref{lem:traj_breakdown}, \ref{lem:E_t_spectral}, \ref{lem:w_eig_ratio}, and Theorem \ref{thm:implicit_bias_gd} correspond to the above four steps, respectively. The reader is referred to the lemmas and theorem for more details.  

\subsection{A Primer on Tensor Algebra and Maintaining Symmetric Property}
We start with the spectral norm of tensors, which resembles the operator norm of matrices \cite{qi2019tensor}.
\begin{definition}
	Given a cubic tensor $\mb{w} \in \RR^{n \circ l}$, its spectral norm $\|\cdot\|_S$ is defined respectively as:
	\begin{align*}
		\|\mb{w}\|_S &= \sup \left\{| \langle \mb{w}, u^{\otimes l} \rangle |\ \|u\|_2 = 1, u \in \RR^{n} \right\}
	\end{align*}
\end{definition}
There are many definitions for tensor eigenvalues \cite{qi2012spectral}, and in this paper we introduce a novel variational characterization of eigenvalues that resembles the Courant-Fisher minimax definition for eigenvalues of matrices, called the v-Eigenvalue. We denote the $i^{th}$ v-Eigenvalue of $\mb{w}$ as $\lambda_i^v(\mb{w})$. Note this is a new definition that is first introduced in this paper and might be of independent interest outside of the current scope.
\begin{definition}[Variational Eigenvalue of Tensors]\label{def:v_eigenvalues}
	For a given tensor $\mb{w} \in \RR^{n \circ l}$, we define its $k^{th}$ variational eigenvalue (v-Eigenvalue) $\lambda_k^v(\mb{w})$ as
	\[
		\lambda_k^v(\mb{w}) \coloneqq \max_{\substack{S \\ \dim(S)=k}} \min_{\mb{u} \in S} \frac{|\langle \mb{w}, \mb{u} \rangle |}{\|\mb{u}\|^2_F}, \quad k \in [n]
	\]
\end{definition}
	where $S$ is a subspace of $\RR^{n \circ l}$ that is spanned by a set of orthogonal, symmetric, rank-1 tensors. Its dimension denotes the number of orthogonal tensors that span this space. It is apparent from the definition that $\|\mb{w}\|_S = \lambda_1^v(\mb{w})$.
	
Next, since most of our analysis relies on the symmetry of the underlying tensor, it is desirable to show that every tensor along the optimization trajectory of GD on \eqref{eq:lifted_main_highr} remains symmetric if started from a symmetric tensor. Please find its proof in Appendix \ref{sec:app_implicit_proofs}.
\begin{lemma}\label{lem:gd_sym}
	If the GD trajectory of \eqref{eq:lifted_main_highr} $\{ \mb{w}_{t}\}_{t=0}^{\infty}$ is initialized at a symmetric rank-1 tensor $\mb{w}_{0}$, then $\{ \mb{w}_{t}\}_{t=0}^{\infty}$ will all be symmetric.
\end{lemma}

\subsection{Main Ideas and Proof Sketch}
In this subsection, we highlight the main ideas behind implicit bias in GD. Lemma \ref{lem:traj_breakdown} and \ref{lem:E_t_spectral} details the first and second step, and are deferred to Appendix \ref{sec:app_implicit_proofs}. The proofs to the results of this section can also be found in that appendix. The lemmas alongside with their proofs are highly technical and not particularly enlightening, therefore omitted here for simplicity. However, the most important takeaway is that for the $t^{th}$ iterate along the GD trajectory of \eqref{eq:lifted_main_highr}, we have the decomposition
\[
	\mb{w}_{t+1} = \langle \mb{Z}_t, \mb{w}_0 \rangle - \mb{E}_t \coloneqq \mb{\tilde w}_t - \mb{E}_t
\]
for some $ \mb{Z}_t$ and $\mb{E}_t$ such that $\|\mb{E}_t\|_S = \mathcal{O}(\epsilon^3)$. This essentially means that by scaling the initialization $\mb{w}_0$ to be small in scale, the error term $\mb{E}_t$ can be ignored from a spectral standpoint, and scales with $\epsilon$ at a cubic rate. This will soon be proven to be useful next.

\begin{lemma}\label{lem:w_eig_ratio}
	Given $\mb{w}_t$ along the GD trajectory of \eqref{eq:lifted_main_highr}, its first two v-eigenvalues, as defined in definition~\ref{def:v_eigenvalues}, satisfy the relation
	\begin{equation}\label{eq:w_eig_separation}
		\frac{\lambda^v_{2}(\mb{ w}_t)}{\lambda^v_1(\mb{ w}_t)} \leq \frac{  \|x_0\|_2^l (1+\eta \sigma_2^l(U))^t + \|\mb{E}_t\|_S/\epsilon}{|v_1^\top x_0|^l(1+\eta \sigma_1^l(U))^t - \|\mb{E}_t\|_S/\epsilon} =\frac{ \|x_0\|_2^l (1+\eta \sigma_2^l(U))^t +  \mathcal{O}(\epsilon^2)}{|v_1^\top x_0|^l(1+\eta \sigma_1^l(U))^t - \mathcal{O}(\epsilon^2)} 
	\end{equation}
	where $\sigma_{1}(U)$ and $\sigma_{2}(U)$ denote the first and second singular values of $U = \langle \mb{A}_r^* \mb{A},M^*\rangle \in \RR^{nr \times nr}$, and $v_1,v_2$ are the associated singular vectors.
\end{lemma}
Lemma~\ref{lem:w_eig_ratio} showcases that when $\epsilon$ is small, the ratio between the largest and second largest v-eigenvalues of $\mb{w}$ is dominated by $(\|x_0\|_2^l (1+\eta \sigma_2^l(U))^t)/(|v_1^\top x_0|^l(1+\eta \sigma_1^l(U))^t)$.

Now, if either $\|x_0\|_2^l$ is large or $|v_1^\top x_0|^l$ approaches 0 in value, then the ratio may be relatively large, contradicting our claim. However, this issue can be easily addressed by letting $x_0 = v_1 +g \in \RR^{nr}$, where $g$ is a vector with each entry being i.i.d sampled from the Gaussian distribution $\mathcal{N}(0,\rho)$. Note that since $U = \langle \mathbf{A}_r, b \rangle_3$, we can calculate $U$ and $v_1$ directly. Lemma~\ref{lem:gauss_init_value} in Appendix \ref{sec:app_implicit_proofs} shows that with this initialization, $|v_1^\top x_0|^l = \mathcal{O}(1)$ and $\|x_0\|_2^l = \mathcal{O}(1)$ with high probability if we select $\rho = \mathcal{O}(1/nr)$. Therefore, the $t^{th}$ iterate along the GD trajectory of \eqref{eq:lifted_main_highr} satisfies
	\begin{equation}\label{eq:approx_eig_ratio}
		\frac{\lambda^v_{2}(\mb{w}_t)}{\lambda^v_1(\mb{w}_t)} \asymp \frac{ (1+\eta \sigma_2^l(U))^t}{(1+\eta \sigma_1^l(U))^t}
	\end{equation}
	with hight probability if $\rho$ is small. This implies that "the level of parametrization helps with separation of eigenvalues", since increasing $l$ will decrease ratio $\lambda^v_{2}(\mb{w}_t)/\lambda^v_1(\mb{w}_t)$. Furthermore, regardless of the value of $\sigma_1(U)$, a larger $t$ will make this ratio exponentially smaller, proving the efficacy of algorithmic regularization of GD in tensor space.
	
	By combining the above facts, we arrive at a major result showing how a small initialization could make the points along the GD trajectory penalize towards rank-1 as $t$ increases
	\begin{theorem}\label{thm:implicit_bias_gd}
		Given the optimization problem \eqref{eq:lifted_main_highr} and its GD trajectory over some finite horizon $T$, i.e.,  $\{\mb{w}_t\}_{t=0}^{T}$ with $\mb{w}_{t+1} = \mb{w}_{t} - \eta \nabla h^l(\mb{w}_{t})$, where $\eta$ is the stepsize, then there exist $t(\kappa,l) \geq 1$ and $\kappa < 1$ such that
		\begin{equation}\label{eq:w_eig_kappa}
			\frac{\lambda^v_{2}(\mb{w}_t)}{\lambda^v_1(\mb{w}_t)} \leq \kappa, \qquad \forall t \in [t(\kappa,l), t_T]
		\end{equation}
		if $\mb{w}_0$ is initialized as $\mb{w}_0 = \epsilon x_0^{\otimes l}$ with a sufficiently small $\epsilon$, where $t(\kappa, l)$ is expressed as
		\begin{equation}\label{eq:t_kappa_l}
			t(\kappa, l) = \left \lceil \ln\left( \frac{\|x_0\|^l_2}{\kappa |v_1^\top x_0|^l}\right) \ln\left( \frac{ 1+\eta \sigma_1^l(U)}{1+\eta \sigma_2^l(U)} \right)^{-1} \right \rceil
		\end{equation}
\end{theorem}
By using the initialization introduced in Lemma~\ref{lem:gauss_init_value}, we can improve the result of Theoerem~\ref{thm:implicit_bias_gd}, which does not need $\epsilon$ to be arbitrarily small. The full details are presented in Corollary~\ref{cor:asymp_implicit_bias} in Appendix \ref{sec:app_implicit_proofs}, stating that as along as $t \asymp \ln\left(1/\kappa \right) \ln\left( (1+\eta \sigma_1^l(U))/(1+\eta \sigma_2^l(U)) \right)^{-1}$, $\mb{w}_t$ will be $\kappa$-rank-1, as long as $\epsilon$ is chosen as a function of $U, r, n, L_s$, and $\kappa$. Note that we say a tensor $\mb{w}$ is "$\kappa$-rank-1" if $\lambda^v_{2}(\mb{w})/\lambda^v_1(\mb{w}) \leq \kappa$.

\section{Approximate Rank-1 Tensors are Benign}

Now that we have established the fact that performing gradient descent on \eqref{eq:lifted_main_highr} will penalize the tensor towards rank-1, it begs the question whether approximate rank-1 tensors can also escape from saddle points, which is the most important question under study in this paper. Please find the proofs to the results in this section in Appendix~\ref{sec:app_rank1_properties}.

To do so, we first introduce a \emph{major spectral} decomposition of symmetric tensors that is helpful.
\begin{proposition}\label{prop:tensor_rank1_breakdown}
	Given a symmetric tensor $\mb{w} \in \RR^{nr \circ l}$, it can be decomposed into two terms, namely a term consisting of its dominant component and another term that is orthogonal to this direction:
	\begin{equation}\label{eq:tensor_rank1_breakdown}
		\mb{w} = \pm \lambda_1^v(\mb{w}) w_s^{\otimes l} + \mb{w}^\dagger \coloneqq \mb{w}_\sigma + \mb{w}^\dagger, \quad w_s \in \RR^{n}, \ \|w_s\|_2=1
	\end{equation}
	where $\langle \mb{w}, w_s^{\otimes l} \rangle = \lambda_1^v(\mb{w})$ and $\langle \mb{w}^\dagger, w_s^{\otimes l} \rangle=0$. Furthermore, if $\mb{w}$ is a $\kappa$-rank-1 tensor, then $\|\mb{w}^\dagger\|_S \leq \kappa \lambda^v_1(\mb{w}_t)$.
\end{proposition}
Next, we characterize the first-order points of \eqref{eq:lifted_main_highr} with approximate rank-1 tensors in mind. Previously, we showed that if a given FOP of \eqref{eq:lifted_main_highr} is symmetric and rank-1, it has a one-to-one correspondence with FOPs of \eqref{eq:unlifted_main}. However, if the FOPs of \eqref{eq:lifted_main_highr} are not exactly rank-1, but instead $\kappa$-rank-1, it is essential to understand whether they maintain the previous properties. This will be addressed below.
\begin{proposition}\label{prop:kappa_fop}
	Assume that a symmetric tensor $\mb{w} \in \RR^{nr \circ l}$ is an FOP of \eqref{eq:lifted_main_highr}, meaning that \eqref{eq:focp_lifted_highr} holds. If it is a $\kappa$-rank-1 tensor with $\kappa \leq \mathcal{O}(1/\|M^*\|^2_F)$, then it admits a decomposition as
	\[
		\mb{w} = \pm \lambda_1^v(\mb{w}) \hat w^{\otimes l} + \mb{w}^\dagger
	\]
	with $\mat(\hat w) \in \RR^{n \times r}$ being an FOP of \eqref{eq:unlifted_main} and $\|\mb{w}^\dagger\|_S \leq \kappa \lambda_1^v(\mb{w})$ by definition.
\end{proposition}
The proposition above asserts that for any given FOP of \eqref{eq:lifted_main_highr}, if it is $\kappa$-rank-1 rather than being truly rank-1, it will consist of a rank-1 term representing a lifted version of an unlifted FOP, as well as a term with a small spectral norm. Referring to \eqref{eq:magnitude_t_w_eig}, it is possible to achieve a significantly low $\kappa$ through a moderate number of iterations. This result, considered the cornerstone of this paper, demonstrates that the use of gradient descent with small initialization will find critical points that are lifted FOPs of \eqref{eq:unlifted_main} with added noise, maintaining a robust association between FOPs of \eqref{eq:lifted_main_highr} and \eqref{eq:unlifted_main}. This finding also facilitates this subsequent theorem:
\begin{theorem}\label{thm:approx_rank1_saddle}
	Assume that a symmetric tensor $\mb{\hat w} \in \RR^{nr \circ l}$ is an FOP of \eqref{eq:lifted_main_highr} that is $\kappa$-rank-1 with $\kappa \leq \mathcal{O}(1/\|M^*\|^2_F)$. Consider its major spectral decomposition $\mb{\hat w} = \lambda_S \hat x^{\otimes l} + \mb{\hat w^\dagger}$ with $\hat x \in \RR^{nr}$,	then it has a rank-1 escape direction if $\hat X  = \mat(\hat x)$ satisfies the inequality
	\begin{equation}\label{eq:main_thm_cond}
		\|M^* - \hat X \hat X^\top \|^2_F \geq \frac{L_s}{\alpha_s} \lambda_r(\hat X \hat X^\top) \tr(M^*) + \mathcal{O}(r \kappa^{1/l})
	\end{equation}
	where $l$ is odd and large enough so that $l > 1/(1-\log_2(2 \beta))$ and $\beta$ is defined as
	\[
		\beta = \frac{L_s \tr(M^*) \lambda_r(\hat X \hat X^\top)}{\alpha_s \|M^* - \hat X \hat X^\top \|^2_F - \mathcal{O}(r \kappa^{1/l})}.
	\]
\end{theorem}
This theorem conveys the message that by running GD on \eqref{eq:lifted_main_highr}, all critical points have escape directions as long as the point is not close to the ground truth solution. In Appendix~\ref{sec:app_highr}, we present Theorem~\ref{thm:socp_corollary} to provide sufficient conditions for the conversion to hold globally when \eqref{eq:main_thm_cond} is hard to hold.

\section{Numerical Experiments}
In this section\footnote{'https://github.com/anonpapersbm/implicit\_bias\_tensor',run on 2021 Macbook Pro}, after we run a given algorithm on \eqref{eq:lifted_main_highr} to completion and obtain a final tensor $\mb{w}_T$, we then apply tensor PCA (detailed in Appendix~\ref{sec:app_algos}) on $\mb{w}_T$ to extract its dominant rank-1 component and recover $X_T \in \RR^{n \times r}$ such that $(\mb{w}_T)_s = \lambda_s \vecc(X_T)^{\otimes l}$. Since $\mb{w}_T$ will be approximately rank-1, the success of this operation is expected \cite{kofidis2002best,wu2020symmetric}. We consider a trial to be successful if the recovered $X_T$ satisfies $\|X_T X_T^\top - M^*\|_F \leq 0.05$. We also initialize our algorithm as per Lemma~\ref{lem:gauss_init_value}.

\subsection{Perturbed Matrix Completion}
The perturbed matrix completion problem is introduced in \cite{yalcin2022semidefinite}, which is a noisy version of classic matrix completion problems. The $\mathcal{A}$ operator is introduced as
\begin{equation} \label{eq:operator-ms}
    \mathcal{A}_{\rho}({\bf M})_{ij} := \begin{cases} {\bf M}_{ij}, & \text{if } (i,j) \in \Omega\\ \rho \mathbf{M}_{ij}, & \text{otherwise} \end{cases},
\end{equation}
where $\Omega$ is a measurement set such that $\Omega = \{ (i,i), (i,2k), (2k,i) | \ \ \forall i \in [n], k \in [ \floor{n/2}]\}$. \cite{yalcin2022semidefinite} has proved that each such instance has $\mathcal{O}(2^{\ceil{n/2}}-2)$ spurious local minima, while it satisfies the RIP property with $\delta_{2r} = (1-\rho)/(1+\rho)$ for some sufficiently small $\rho$. This implies that common first-order methods fail with high probability for this class of problems. In our experiment, we apply both lifted and unlifted formulations to \eqref{eq:operator-ms} with $\rho = 0.01$, yielding $\delta_{2r} \approx 1$. We test different values of $n$ and $\epsilon$, using a lifted level of $l=3$. We ran 10 trials each to calculate success rate. If unspecified in the plot, we default $n=10$, $\epsilon=10^{-7}$. Figure~\ref{fig:pmc_1} reveals a higher success rate for the lifted formulation across different problem sizes, with smaller problems performing better as expected (since larger problems require a higher lifting level). Success rates improve with smaller $\epsilon$, emphasizing the importance of small initialization. We employed customGD, a modified gradient descent algorithm with heuristic saddle escaping. This algorithm will deterministically escape from critical points utilizing knowledge from the proof of Theorem~\ref{thm:socp_highr}. For details please refer to Appendix \ref{sec:app_algos}. Furthermore, to showcase the implicit penalization affects of GD, we obtained approximate measures for $\lambda^v_{2}(\mb{w}_t)/\lambda^v_1(\mb{w}_t)$ (since exactly solving for them is NP-hard) along the trajectory, and presented the results and methods in Appendix~\ref{sec:app_exp}.

\begin{figure}[ht]
    \centering
    \includegraphics[width=0.35\linewidth]{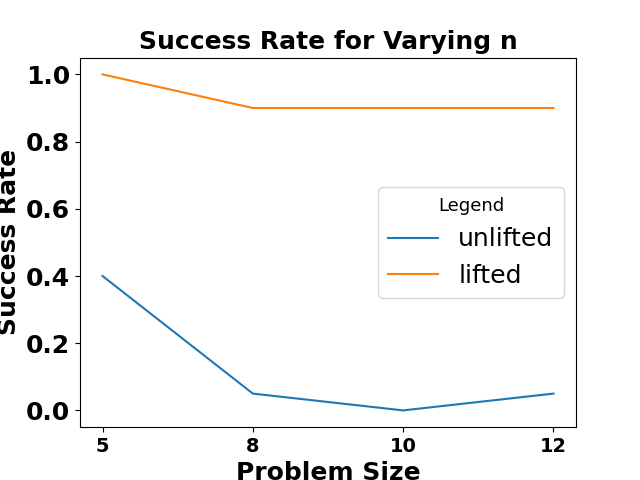}
    \includegraphics[width=0.35\linewidth]{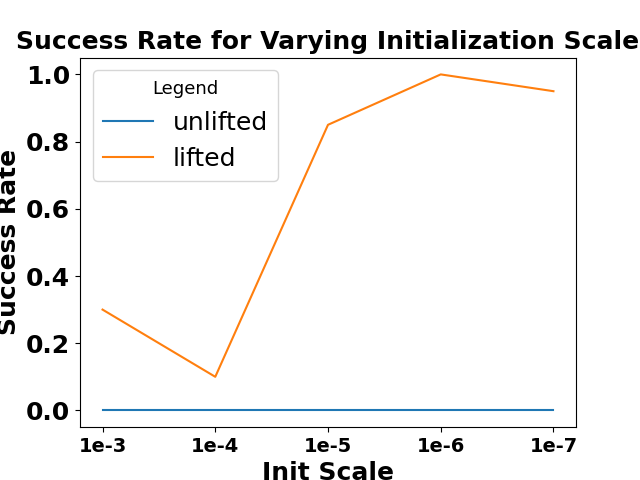}
    \caption{Success rate of the lifted formulation versus the unlifted formulation against varying $n$ and $\epsilon$. The algorithm of choice is CustomGD (details in Appendix \ref{sec:app_algos}).}
    \label{fig:pmc_1}
\end{figure}

Additionally, we examine different algorithms for \eqref{eq:lifted_main_highr}, including customGD, vanilla GD, perturbed GD (\cite{jin2021nonconvex}, for its ability to escape saddles), and ADAM \cite{kingma2014adam}. Figure~\ref{fig:pmc_2} suggest that ADAM is an effective optimizer with a high success rate and rapid convergence, indicating that momentum acceleration may not hinder implicit regularization and warrants further research. Perturbed GD performed poorly, possibly due to random noise disrupting rank-1 penalization.

\begin{figure}[htbp]
    \centering
    \includegraphics[width=0.35\linewidth]{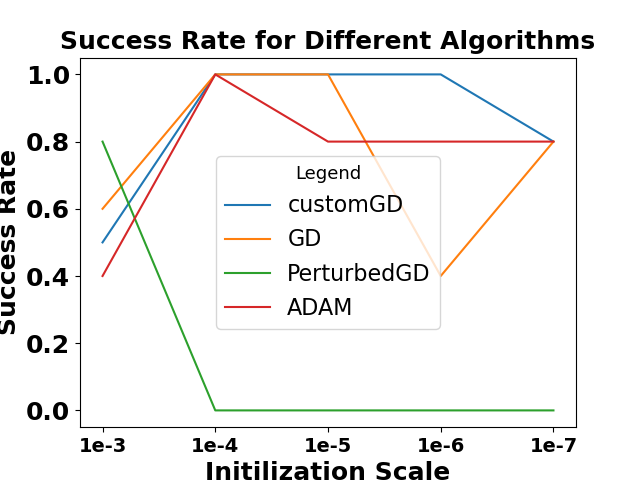}
    \includegraphics[width=0.35\linewidth]{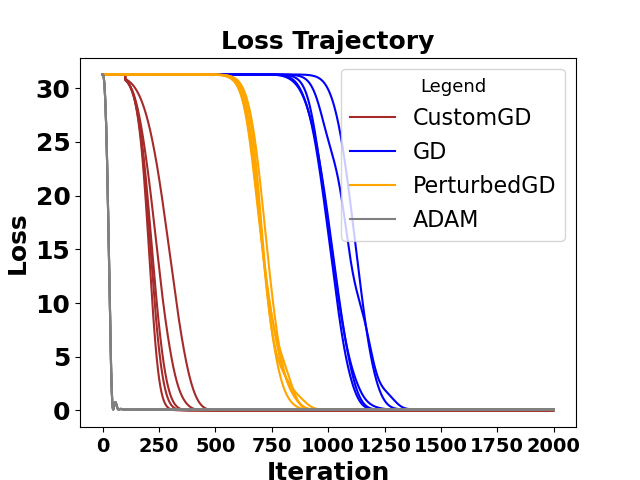}
    \caption{Performance of different algorithms applied to the lifted formulation \eqref{eq:lifted_main_highr}. }
    \label{fig:pmc_2}
\end{figure}


\subsection{Shallow Neural Network Training with Quadratic Activation}
It has long been known that the matrix sensing problem \eqref{eq:unlifted_main} includes the training of two-layer neural networks (NN) with quadratic activation as a special case \cite{li2018algorithmic}. In summary, the output of the neural network $y \in \RR^m$ with respect to $m$ inputs $\{d_i\}_{i=1}^m \in \RR^n$ can be expressed as $y_i = \mathbf{1}^\top q(X^\top d_i)$, which implies $y_i = \langle d_i d_i^\top, XX^\top \rangle$, where $q(\cdot)$ is the element-wise quadratic function and $X \in \RR^{n \times r}$ in \eqref{eq:unlifted_main} represents the weights of the neural network. Thus $r$ represents the number of hidden neurons. In our experiment, we demonstrate that when $m$ is small, the lifted framework \eqref{eq:lifted_main_highr} outperforms standard neural network training in success rate, yielding improved recovery of the true weights. We set the hidden neurons number to be $n$ for the standard network training, thereby comparing the existing over-parametrization framework (recall Section~\ref{sec:related}, with $r_{\text{search}}=n$) with the lifted one . We employ the ADAM optimizer for both methods. Table~\ref{tab:compare} showcases the success rate under various problem and sample sizes. Sampling both data and true weights $Z \in \RR^{n \times r}$ from an i.i.d Gaussian distribution, we calculate the observations $y$ and attempt to recover $Z$ using both approaches. As the number of samples increases, so does the success rate, with the lifted approach offering significantly better accuracy overall, even when the standard training has a 0\% success rate.

\begin{table}[h]
    \centering
    \begin{tabular}{@{}c@{\hspace{1cm}}c@{}}
        \begin{minipage}{0.45\linewidth}
            \centering
            \begin{tabular}{|c|c|c|c|}
                \hline
                Success Rate & m = 20 & m=30 & m=40 \\ \cline{1-4}
            n=8 & 0.9(0) & 1(0.3) & 0.9(0.5) \\ \cline{1-4}
            n=10 & 0.2(0) & 0.6(0) & 0.8(0) \\ \cline{1-4}
            n=12 & 0.1(0) & 0.4(0) & 0.8(0) \\ \cline{1-4}
            \end{tabular}
            \subcaption{Ground truth weight with $r=1$}
        \end{minipage}
        &
        \begin{minipage}{0.45\linewidth}
            \centering
            \begin{tabular}{|c|c|c|c|}
                \hline
                Success Rate & m = 30 & m=40 & m=50 \\ \cline{1-4}
            n=8 & 0.3(0) & 0.3(0) & 0.8(0) \\ \cline{1-4}
            n=10 & 0.3(0) & 0.4(0) & 0.2(0) \\ \cline{1-4}
            n=12 & 0(0) & 0(0) & 0.2(0) \\ \cline{1-4}
            \end{tabular}
            \subcaption{Ground truth weight with $r=2$}
        \end{minipage}
    \end{tabular}
    \captionof{table}{Success rate of NN training using  \eqref{eq:lifted_main_highr} and original formulation. The number inside the parentheses denotes the success rate of the original formulations. $\epsilon = 10^{-5}$ and $l=3$.}
    \label{tab:compare}
\end{table}

\section{Conclusion}
Our study highlights the pivotal role of gradient descent in inducing implicit regularization within tensor optimization, specifically in the context of the lifted matrix sensing framework. We reveal that GD can lead to approximate rank-1 tensors and critical points with escape directions when initialized at an adequately small scale. This work also contributes to the usage of tensors in machine learning models, as we introduce novel concepts and techniques to cope with the intrinsic complexities of tensors. 
\newpage
\section{Acknowledgement}
This work was supported by grants from ARO, ONR, AFOSR, NSF, and the UC Noyce Initiative.
\bibliography{references}

\begin{thebibliography}{10}

\bibitem{donoho2006compressed}
D.~L. Donoho, ``Compressed sensing,'' {\em IEEE Transactions on information
  theory}, vol.~52, no.~4, pp.~1289--1306, 2006.

\bibitem{candes2009exact}
E.~J. Cand{\`e}s and B.~Recht, ``Exact matrix completion via convex
  optimization,'' {\em Foundations of Computational Mathematics}, vol.~9,
  no.~6, pp.~717--772, 2009.

\bibitem{candes2010power}
E.~J. Cand{\`e}s and T.~Tao, ``The power of convex relaxation: Near-optimal
  matrix completion,'' {\em IEEE Transactions on Information Theory}, vol.~56,
  no.~5, pp.~2053--2080, 2010.

\bibitem{nguyen2019low}
L.~T. Nguyen, J.~Kim, and B.~Shim, ``Low-rank matrix completion: A contemporary
  survey,'' {\em IEEE Access}, vol.~7, pp.~94215--94237, 2019.

\bibitem{singer2011angular}
A.~Singer, ``Angular synchronization by eigenvectors and semidefinite
  programming,'' {\em Applied and Computational Harmonic Analysis}, vol.~30,
  no.~1, pp.~20--36, 2011.

\bibitem{boumal2016nonconvex}
N.~Boumal, ``Nonconvex phase synchronization,'' {\em SIAM Journal on
  Optimization}, vol.~26, no.~4, pp.~2355--2377, 2016.

\bibitem{shechtman2015phase}
Y.~Shechtman, Y.~C. Eldar, O.~Cohen, H.~N. Chapman, J.~Miao, and M.~Segev,
  ``Phase retrieval with application to optical imaging: A contemporary
  overview,'' {\em IEEE Signal Processing Magazine}, vol.~32, no.~3,
  pp.~87--109, 2015.

\bibitem{fattahi2020exact}
S.~Fattahi and S.~Sojoudi, ``Exact guarantees on the absence of spurious local
  minima for non-negative rank-1 robust principal component analysis,'' {\em
  Journal of Machine Learning Research}, vol.~21, pp.~1--51, 2020.

\bibitem{zhang2017conic}
Y.~Zhang, R.~Madani, and J.~Lavaei, ``Conic relaxations for power system state
  estimation with line measurements,'' {\em IEEE Transactions on Control of
  Network Systems}, vol.~5, no.~3, pp.~1193--1205, 2017.

\bibitem{jin2019towards}
M.~Jin, I.~Molybog, R.~Mohammadi-Ghazi, and J.~Lavaei, ``Towards robust and
  scalable power system state estimation,'' in {\em 2019 IEEE 58th Conference
  on Decision and Control (CDC)}, pp.~3245--3252, IEEE, 2019.

\bibitem{molybog2020conic}
I.~Molybog, R.~Madani, and J.~Lavaei, ``Conic optimization for quadratic
  regression under sparse noise,'' {\em The Journal of Machine Learning
  Research}, vol.~21, no.~1, pp.~7994--8029, 2020.

\bibitem{li2018algorithmic}
Y.~Li, T.~Ma, and H.~Zhang, ``Algorithmic regularization in over-parameterized
  matrix sensing and neural networks with quadratic activations,'' in {\em
  Conference On Learning Theory}, pp.~2--47, PMLR, 2018.

\bibitem{burer2003nonlinear}
S.~Burer and R.~D. Monteiro, ``A nonlinear programming algorithm for solving
  semidefinite programs via low-rank factorization,'' {\em Mathematical
  Programming}, vol.~95, no.~2, pp.~329--357, 2003.

\bibitem{zhang2019sharp}
R.~Y. Zhang, S.~Sojoudi, and J.~Lavaei, ``Sharp restricted isometry bounds for
  the inexistence of spurious local minima in nonconvex matrix recovery,'' {\em
  Journal of Machine Learning Research}, vol.~20, no.~114, pp.~1--34, 2019.

\bibitem{ma2022sharp}
Z.~Ma, Y.~Bi, J.~Lavaei, and S.~Sojoudi, ``Sharp restricted isometry property
  bounds for low-rank matrix recovery problems with corrupted measurements,''
  in {\em Proceedings of the AAAI Conference on Artificial Intelligence},
  vol.~36, pp.~7672--7681, 2022.

\bibitem{ha2020equivalence}
W.~Ha, H.~Liu, and R.~F. Barber, ``An equivalence between critical points for
  rank constraints versus low-rank factorizations,'' {\em SIAM Journal on
  Optimization}, vol.~30, no.~4, pp.~2927--2955, 2020.

\bibitem{zhang2021general}
H.~Zhang, Y.~Bi, and J.~Lavaei, ``General low-rank matrix optimization:
  Geometric analysis and sharper bounds,'' {\em Advances in Neural Information
  Processing Systems}, vol.~34, pp.~27369--27380, 2021.

\bibitem{zhang2021sharp}
R.~Y. Zhang, ``Sharp global guarantees for nonconvex low-rank matrix recovery
  in the overparameterized regime,'' {\em arXiv preprint arXiv:2104.10790},
  2021.

\bibitem{zhang2022improved}
R.~Y. Zhang, ``Improved global guarantees for the nonconvex burer--monteiro
  factorization via rank overparameterization,'' {\em arXiv preprint
  arXiv:2207.01789}, 2022.

\bibitem{yalcin2022semidefinite}
B.~Yalcin, Z.~Ma, J.~Lavaei, and S.~Sojoudi, ``Semidefinite programming versus
  burer-monteiro factorization for matrix sensing,'' in {\em Proceedings of the
  AAAI Conference on Artificial Intelligence}, 2023.

\bibitem{ma2023over}
Z.~Ma, I.~Molybog, J.~Lavaei, and S.~Sojoudi, ``Over-parametrization via
  lifting for low-rank matrix sensing: Conversion of spurious solutions to
  strict saddle points,'' in {\em International Conference on Machine
  Learning}, PMLR, 2023.

\bibitem{lasserre2001global}
J.~B. Lasserre, ``Global optimization with polynomials and the problem of
  moments,'' {\em SIAM Journal on optimization}, vol.~11, no.~3, pp.~796--817,
  2001.

\bibitem{kofidis2002best}
E.~Kofidis and P.~A. Regalia, ``On the best rank-1 approximation of
  higher-order supersymmetric tensors,'' {\em SIAM Journal on Matrix Analysis
  and Applications}, vol.~23, no.~3, pp.~863--884, 2002.

\bibitem{wu2020symmetric}
L.~Wu, X.~Liu, and Z.~Wen, ``Symmetric rank-1 approximation of symmetric
  high-order tensors,'' {\em Optimization Methods and Software}, vol.~35,
  no.~2, pp.~416--438, 2020.

\bibitem{levin2022effect}
E.~Levin, J.~Kileel, and N.~Boumal, ``The effect of smooth parametrizations on
  nonconvex optimization landscapes,'' {\em arXiv preprint arXiv:2207.03512},
  2022.

\bibitem{li2022survey}
P.~Li, X.~Liang, and H.~Song, ``A survey on implicit bias of gradient
  descent,'' in {\em 2022 14th International Conference on Computer Research
  and Development (ICCRD)}, pp.~108--114, IEEE, 2022.

\bibitem{stoger2021small}
D.~St{\"o}ger and M.~Soltanolkotabi, ``Small random initialization is akin to
  spectral learning: Optimization and generalization guarantees for
  overparameterized low-rank matrix reconstruction,'' {\em Advances in Neural
  Information Processing Systems}, vol.~34, pp.~23831--23843, 2021.

\bibitem{jin2023understanding}
J.~Jin, Z.~Li, K.~Lyu, S.~S. Du, and J.~D. Lee, ``Understanding incremental
  learning of gradient descent: A fine-grained analysis of matrix sensing,''
  {\em arXiv preprint arXiv:2301.11500}, 2023.

\bibitem{ma2022global}
J.~Ma and S.~Fattahi, ``Global convergence of sub-gradient method for robust
  matrix recovery: Small initialization, noisy measurements, and
  over-parameterization,'' {\em arXiv preprint arXiv:2202.08788}, 2022.

\bibitem{ma2023geometric}
Z.~Ma, Y.~Bi, J.~Lavaei, and S.~Sojoudi, ``Geometric analysis of noisy low-rank
  matrix recovery in the exact parametrized and the overparametrized regimes,''
  {\em INFORMS Journal on Optimization}, 2023.

\bibitem{recht2010guaranteed}
B.~Recht, M.~Fazel, and P.~A. Parrilo, ``Guaranteed minimum-rank solutions of
  linear matrix equations via nuclear norm minimization,'' {\em SIAM Review},
  vol.~52, no.~3, pp.~471--501, 2010.

\bibitem{cai2013sharp}
T.~T. Cai and A.~Zhang, ``Sharp rip bound for sparse signal and low-rank matrix
  recovery,'' {\em Applied and Computational Harmonic Analysis}, vol.~35,
  no.~1, pp.~74--93, 2013.

\bibitem{zhuo2021computational}
J.~Zhuo, J.~Kwon, N.~Ho, and C.~Caramanis, ``On the computational and
  statistical complexity of over-parameterized matrix sensing,'' {\em arXiv
  preprint arXiv:2102.02756}, 2021.

\bibitem{razin2021implicit}
N.~Razin, A.~Maman, and N.~Cohen, ``Implicit regularization in tensor
  factorization,'' in {\em International Conference on Machine Learning},
  pp.~8913--8924, PMLR, 2021.

\bibitem{razin2022implicit}
N.~Razin, A.~Maman, and N.~Cohen, ``Implicit regularization in hierarchical
  tensor factorization and deep convolutional neural networks,'' in {\em
  International Conference on Machine Learning}, pp.~18422--18462, PMLR, 2022.

\bibitem{ge2021understanding}
R.~Ge, Y.~Ren, X.~Wang, and M.~Zhou, ``Understanding deflation process in
  over-parametrized tensor decomposition,'' {\em Advances in Neural Information
  Processing Systems}, vol.~34, pp.~1299--1311, 2021.

\bibitem{candes2011tight}
E.~J. Candes and Y.~Plan, ``Tight oracle inequalities for low-rank matrix
  recovery from a minimal number of noisy random measurements,'' 2011.

\bibitem{qi2019tensor}
L.~Qi, S.~Hu, X.~Zhang, and Y.~Chen, ``Tensor norm, cubic power and gelfand
  limit,'' {\em arXiv preprint arXiv:1909.10942}, 2019.

\bibitem{qi2012spectral}
L.~Qi, ``The spectral theory of tensors (rough version),'' {\em arXiv preprint
  arXiv:1201.3424}, 2012.

\bibitem{jin2021nonconvex}
C.~Jin, P.~Netrapalli, R.~Ge, S.~M. Kakade, and M.~I. Jordan, ``On nonconvex
  optimization for machine learning: Gradients, stochasticity, and saddle
  points,'' {\em Journal of the ACM (JACM)}, vol.~68, no.~2, pp.~1--29, 2021.

\bibitem{kingma2014adam}
D.~P. Kingma and J.~Ba, ``Adam: A method for stochastic optimization,'' {\em
  arXiv preprint arXiv:1412.6980}, 2014.

\bibitem{li2019non}
Q.~Li, Z.~Zhu, and G.~Tang, ``The non-convex geometry of low-rank matrix
  optimization,'' {\em Information and Inference: A Journal of the IMA},
  vol.~8, no.~1, pp.~51--96, 2019.

\bibitem{kolda2015numerical}
T.~G. Kolda, ``Numerical optimization for symmetric tensor decomposition,''
  {\em Mathematical Programming}, vol.~151, no.~1, pp.~225--248, 2015.

\bibitem{petersen2008matrix}
K.~B. Petersen, M.~S. Pedersen, {\em et~al.}, ``The matrix cookbook,'' {\em
  Technical University of Denmark}, vol.~7, no.~15, p.~510, 2008.

\bibitem{ma2023noisy}
Z.~Ma and S.~Sojoudi, ``Noisy low-rank matrix optimization: Geometry of local
  minima and convergence rate,'' in {\em International Conference on Artificial
  Intelligence and Statistics}, pp.~3125--3150, PMLR, 2023.

\bibitem{zhang2020many}
G.~Zhang and R.~Y. Zhang, ``How many samples is a good initial point worth in
  low-rank matrix recovery?,'' in {\em Advances in Neural Information
  Processing Systems}, vol.~33, pp.~12583--12592, 2020.

\bibitem{bi2020global}
Y.~Bi and J.~Lavaei, ``Global and local analyses of nonlinear low-rank matrix
  recovery problems,'' 2020.
\newblock arXiv:2010.04349.

\bibitem{lim2013blind}
L.-H. Lim and P.~Comon, ``Blind multilinear identification,'' {\em IEEE
  Transactions on Information Theory}, vol.~60, no.~2, pp.~1260--1280, 2013.

\bibitem{comon2008symmetric}
P.~Comon, G.~Golub, L.-H. Lim, and B.~Mourrain, ``Symmetric tensors and
  symmetric tensor rank,'' {\em SIAM Journal on Matrix Analysis and
  Applications}, vol.~30, no.~3, pp.~1254--1279, 2008.

\bibitem{ni2019hermitian}
G.~Ni, ``Hermitian tensor and quantum mixed state,'' {\em arXiv preprint
  arXiv:1902.02640}, 2019.

\bibitem{chang2021hanson}
S.~Y. Chang, ``Hanson-wright inequality for random tensors under einstein
  product,'' {\em arXiv preprint arXiv:2111.12169}, 2021.

\bibitem{vershynin2018high}
R.~Vershynin, {\em High-dimensional probability: An introduction with
  applications in data science}, vol.~47.
\newblock Cambridge university press, 2018.

\end{thebibliography}
\newpage
\appendix

\section{Additional Definitions and Supporting Lemmas}
\subsection{Notations} \label{sec:app_notations}
In this paper, $\sigma_i(M)$ denotes the $i$-th largest singular value of a matrix $M$, and $\lambda_i(M)$ denotes the $i$-th largest eigenvalue of $M$. $\norm{v}$ denotes the Euclidean norm of a vector $v$, while $\norm{M}_F$ and $\norm{M}_2$ denote the Frobenius norm and induced $l_2$ norm of a matrix $M$, respectively. For a matrix $M$, $\vecc(M)$ is the usual vectorization operation by stacking the columns of the matrix $M$ into a vector. For a vector $v \in \RR^{n^2}$, $\mat(v)$ converts $v$ to a square matrix and $\mat_S(v)$ converts $v$ to a symmetric matrix, i.e., $\mat(v)=M$ and $\mat_S(v) = (M+M^T)/2$, where $M \in \RR^{n \times n}$ is the unique matrix satisfying $v=\vecc(M)$. $[n]$ denotes the integer set $[1,\dots,n]$, and $\circ l$ stands for the shorthand of repeated cartesian product $\times \dots \times$ for $l$ times. The symbol $\oslash$ denotes the kronecker product, while $\otimes$ denotes tensor outer product. $\asymp$ denotes "asymptotic to", meaning that the two terms on both sides of this symbol have the same order of magnitude.

\subsection{Critical Conditions for Unlifted Problem}\label{sec:app_fopsop}
We present the FOP and SOP conditions for the unlifted problem as our benchmark.

\begin{lemma}\label{lem:cp_unlifted}
	The vector $\hat X \in \RR^{n \times r}$ is an SOP of \eqref{eq:unlifted_main} if and only if
	\begin{gather}\label{eq:focp_unlifted}
		\nabla f(\hat X \hat X^\top)\hat X = 0, \\
		\label{eq:socp_unlifted}
		2 \langle \nabla f(\hat X \hat X^\top), U U^\top \rangle + [\nabla^2 f(\hat X \hat X^\top)](\hat X U^\top + U \hat X^\top,\hat X U^\top + U \hat X^\top) \geq 0 \quad \forall U \in \RR^{n \times r}
	\end{gather}
	with \eqref{eq:focp_unlifted} being the necessary and sufficient condition for $\hat X$ to be an FOP. 
\end{lemma}
A proof to the above lemma can be found in many matrix sensing literatures, including \cite{ha2020equivalence,zhang2021general,li2019non}, etc.

\subsection{Additional Definitions} \label{sec:app_def}
\begin{definition}[RIP, \citep{candes2009exact}]  \label{def:rip}
    Given a natural number $p$, the linear map $\mathcal{A}: \mathbb{R}^{n \times n} \mapsto \mathbb{R}^{m}$ is said to satisfy $\delta_{p}$-RIP if there is a constant $\delta_{p} \in [0,1)$ such that
    \[ (1-\delta_{p})\|M \|_F^2 \leq \| \mathcal{A} (M) \|^2 \leq (1 + \delta_{p}) \|M\|_F^2 \]
    holds for all matrices $M \in \mathbb{R}^{n \times n}$ satisfying $\rk(M) \leq p$.
\end{definition}

\begin{definition}[Symmetric Tensor]
	Similar to the definition of symmetric matrices, for an order-$l$ tensor $\mb{a}$ with the same dimensions (i.e., $n_1 = \dots = n_l$), also called a cubic tensor, it is said that the tensor is symmetric if  its entries are invariance under any permutation of their indices:
	\[
		a_{i_{\sigma(1)} \cdots i_{\sigma(l)}} = a_{i_1 \cdots i_l} \quad \forall \sigma, \quad i_1, \dots, i_l \in \{1,\dots,n\}
	\]
	where $\sigma \in \mathcal{G}_l$ denotes a specific permutation and $\mathcal{G}_l$ is the symmetric group of permutations on $\{1, \dots, l\}$. We denote the set of symmetric tensors as $\mathrm{S}^l(\RR^n)$.
\end{definition}

\begin{definition}[Rank of Tensors]
	The rank of a cubic tensor $\mb{a} \in \RR^{n \circ l}$ is defined as
	\[
		\rk(\mb{a}) = \min\{r | \mb{a} = \sum_{i=1}^r u_i \otimes v_i \otimes \cdots \otimes w_i \}
	\]
	for some vector $u_i, \dots, w_i \in \RR^n$. Furthermore, according to \cite{kolda2015numerical}, if $\mb{a}$ is a symmetric tensor, then it can be decomposed as:
	\[
		\mb{a} = \sum_{i=1}^r \lambda_i u_i \otimes \dots \otimes u_i \coloneqq \sum_{i=1}^r \lambda_i u_i^{\otimes l}
	\]
	and the rank is conveniently defined as the number of nonzero $\lambda_i$'s, which is very similar to the rank of symmetric matrices indeed. The most important concept in our paper is rank-1 tensors, and for any tensor $\mb{a}$, a necessary and sufficient condition for it to be rank-1 is that 
	\[
		\mb{a} = u^{\otimes l}
	\]
	for some $u \in \RR^n$.
	
\end{definition}

\begin{definition}[Tensor Multiplication]
	Outer product is an operation carried out on a pair  of tensors, denoted as $\otimes$. The outer product of 2 tensors $\mathbf{a}$ and $\mathbf{b}$, respectively of orders $l$ and $p$, is a tensor of order $l+p$, denoted as $\mb{c} = \mb{a} \otimes \mb{b}$ such that:
	\[
		c_{i_1\dots i_l j_1 \dots j_p} = a_{i_1\dots i_l} b_{j_1 \dots j_p}
	\]
	When the 2 tensors are of the same dimension, this product is such that $\otimes: \RR^{n \circ l} \times \RR^{n \circ p} \mapsto \RR^{n \circ (l+p)}$. Henceforth, we use the shorthand notation
	\[
		 \underbrace{a \otimes \dots \otimes a}_{l \ \text{times}}\coloneqq  a^{\otimes l}
	\]
	We also define an inner product of two tensors. The mode-$q$ inner product between the 2 aforementioned tensors having the same $q$-th dimension is denoted as $\langle \mb{a}, \mb{b} \rangle_q$. Without loss of generality, assume that $q = 1$ and 
	\[
		\left[ \langle \mb{a}, \mb{b} \rangle_q \right]_{i_2\dots i_l j_2 \dots j_p} = \sum_{\alpha=1}^{n_q} a_{\alpha i_2\dots i_l} b _{\alpha j_2 \dots j_p}
	\]
	Note that when we write $\langle \cdot, \cdot \rangle_q$, we count the $q$-th dimension of the first entry. Indeed, this definition of inner product can also be trivially extended to multi-mode inner products by just summing over all modes, denoted as $\langle \mb{a}, \mb{b} \rangle_{q, \dots, s}$.
\end{definition}

\subsection{Technical Lemmas}\label{sec:app_lemmas}
\begin{lemma}[Section 10.2 \cite{petersen2008matrix}]
For four arbitrary matrices $A,B,C,D$ of compatible dimensions, it holds that

\begin{equation}
	\langle A \otimes B, C\otimes D \rangle_{2,4} = AC \otimes BD
\end{equation}
	
\end{lemma} \label{lem:kron_iden}

\begin{lemma}[\cite{ma2023noisy}]\label{lem:G_upper}
	For any SOP $\hat X$ of \eqref{eq:unlifted_main}, define $G$ as $G \coloneqq - \lambda_{\text{min}}(\nabla f(\hat X \hat X^\top))$, and $L_s$ be the RSS constant. Then it holds that
	\[
		G \leq \lambda_r(\hat X \hat X^\top) L_s
	\]
	where $r$ is the search rank of \eqref{eq:unlifted_main}.
\end{lemma}

	\begin{lemma}\label{lem:xhat_upper}
		Given an FOP $\hat X$ of \eqref{eq:unlifted_main}, it holds that
		\begin{equation}\label{eq:xhat_upper}
			\lambda_r(\hat X \hat X^\top) < \sqrt{\frac{2 L_s}{r \alpha_s}} \|M^*\|_F
		\end{equation}
\end{lemma}
\begin{proof}[Proof of Lemma \ref{lem:xhat_upper}]
	Lemma 6 of \cite{zhang2021general} states that given an arbitrary constant $\lambda$ and matrix $X \in \RR^{n \times r}$, one can write
	\[
		\|X X^\top\|^2_F \geq \max \left \{\frac{2L_s}{\alpha_s} \|M^*\|^2_F, (\frac{2\lambda \sqrt{r}}{\alpha_s})^{4/3} \right \} \implies \|\nabla h(X)\|_F \geq \lambda
	\]
	A simple negation to both sides gives
	\[
		\|\nabla h(X)\|_F < \lambda \implies \|X X^\top\|^2_F < \max\{\frac{2L_s}{\alpha_s} \|M^*\|^2_F, (\frac{2\lambda \sqrt{r}}{\alpha_s})^{4/3}\}
	\]
	If we set $X = \hat X$, then left-hand side of the above inequality is automatically satisfied for small values of $\lambda$ since $\|\nabla h(\hat X)\|_F = 0$, and thus we conclude that
	\[
		\|\hat X \hat X^\top\|^2_F < \frac{2L_s}{\alpha_s} \|M^*\|^2_F
	\]
	since $(\frac{2\lambda \sqrt{r}}{\alpha_s})^{4/3}$ can be made arbitrarily small. Therefore,
	\[
		\|\hat X \hat X^\top\|^2_F \geq r \lambda_r(\hat X \hat X^\top)^2 \implies \lambda_r(\hat X \hat X^\top) < \sqrt{\frac{2 L_s}{r \alpha_s}} \|M^*\|_F
	\]
	as $\hat X \hat X^\top$ can have at most $r$ eigenvalues due to its factorized form.
\end{proof}

\newpage
\section{Additional Details for Lifted Formulation of General $r$}\label{sec:app_highr}
We analyze \eqref{eq:lifted_main_highr} and generalize the results of \cite{ma2023over} to $r>1$. We start with the characterization of FOPs and SOPs of \eqref{eq:lifted_main_highr}.
\begin{lemma}\label{lem:cp_lifted_highr}
	The tensor $ \mb{\hat w} \in \RR^{nr \circ l}$ is an SOP of \eqref{eq:lifted_main_highr} if and only if
	\begin{subequations}
		\begin{gather}
			\langle \nabla f^l( \langle \mb{P}(\mb{\hat w}), \mb{P}(\mb{\hat w}) \rangle_{2*[l]}) , \mb{P}(\mb{\hat w})\rangle_{2*[l]} = 0, \label{eq:focp_lifted_highr}\\
		\begin{split}
			&2 \langle \nabla f^l(\langle \mb{P}(\mb{\hat w}), \mb{P}(\mb{\hat w}) \rangle_{2*[l]}), \langle \mb{P}(\Delta), \mb{P}(\Delta) \rangle_{2*[l]} + \\
			\| \langle \mb{A}^{\otimes l} , \langle \mb{P}(\mb{\hat w}),& \mb{P}(\Delta) \rangle_{2*[l]} + \langle \mb{P}(\Delta), \mb{P}(\mb{\hat w}) \rangle_{2*[l]} \rangle \|_F^2 \geq 0 \quad \forall \Delta \in \RR^{nr \circ l}
		\end{split}
		\label{eq:socp_lifted_highr}
		\end{gather}
	\end{subequations}
	with \eqref{eq:socp_lifted_highr} being a necessary and sufficient condition for $\mb{\hat w}$ to be a FOP.
\end{lemma}	
\begin{proof}[Proof of Lemma~\ref{lem:cp_lifted_highr}]
	We have
	\begin{equation}\label{eq:lem_cp_highr_help1}
		\nabla f^l( \mb{M}) =  \langle \langle \mb{A}^{\otimes l}, \mb{M} - \mathcal{M}(\vecc(Z)^{\otimes l})\rangle, \mb{A}^{\otimes l} \rangle_{1,4,\dots,3l-2}
	\end{equation}
	where the new map $\mathcal{M}: \RR^{nr \circ l} \mapsto \RR^{n \circ 2l}$ is defined as
	\begin{equation*}
		\mathcal{M}(\mb{w}) = \langle \mb{P}(\mb{w}), \mb{P}(\mb{w}) \rangle_{2*[l]},
	\end{equation*}
	and its total derivative at $\mb{w}$ is the linear map $D_{\mb{w}} \mathcal{M}: \RR^{nr \circ l} \mapsto \RR^{n \circ 2l}$ given below:
	\begin{equation}\label{eq:lem_cp_highr_help2}
		D_{\mb{w}} \mathcal{M}(\mb{v}) = \langle \mb{P}(\mb{v}), \mb{P}(\mb{w}) \rangle_{2*[l]} + \langle \mb{P}(\mb{w}), \mb{P}(\mb{v}) \rangle_{2*[l]}.
	\end{equation}
	Combining \eqref{eq:lem_cp_highr_help1} and \eqref{eq:lem_cp_highr_help2} gives that
	\begin{equation}
		D_{\mb{w}} h^l(\mb{v}) = \langle \mb{A}^{\otimes l}, D_{\mb{w}} \mathcal{M}(\mb{v}) \rangle^\top \langle \mb{A}^{\otimes l}, \mathcal{M}(\mb{w}) - \mathcal{M}(\vecc(Z)^{\otimes l})\rangle
	\end{equation}
	The sensing matrices $A_k \ \forall k \in [m]$ are assumed to be symmetric, and therefore $\langle \mb{A}^{\otimes l}, D_{\mb{w}} \mathcal{M}(\mb{v}) \rangle = 2 \langle \mb{A}^{\otimes l}, \langle \mb{P}(\mb{v}), \mb{P}(\mb{w}) \rangle_{2*[l]} \rangle$. 
	
	Therefore, since the first-order optimality condition for \eqref{eq:lifted_main_highr} is that $D_{\mb{w}} h^l(\mb{v}) = 0 \ \forall \mb{v} \in \RR^{nr \circ l}$, it can be equivalently written as 
	\begin{equation}\label{eq:highr_grad_exp}
		\langle \langle \mb{A}^{\otimes l}, \mb{P}(\mb{w})\rangle_{2*[l]}, \langle \mb{A}^{\otimes l}, \mathcal{M}(\mb{w}) - \mathcal{M}(\vecc(Z)^{\otimes l})\rangle \rangle_{1,3,\dots,2l-1} = 0,
	\end{equation}
	and left-hand side of the above equation yields \eqref{eq:focp_lifted_highr} after rearrangements.
	
	For the second-order optimality condition, one can directly take the derivative of $D_{\mb{w}} h^l(\mb{v})$, but there is an easier way since we are only concerned the expression of its quadratic form evaluated at some tensor $\Delta \in \RR^{nr \circ l}$. For a brief moment, assume that we aim to optimize over $\mb{X} \in \RR^{[n \times r] \circ l}$, for which
	\[
		\nabla h^l(\mb{X})  = 2 \langle \nabla f^l( \langle \mb{X}, \mb{X} \rangle_{2*[l]}) , \mb{X}\rangle_{2*[l]} \in \RR^{[n \times r] \circ l}
	\]
	Therefore, if we instead take the derivate of $g(\mb{P}(\mb{w}))$ with respect to $\mb{w}$,  we can simply use the chain rule and arrive at
	\begin{equation}\label{eq:lem_cp_highr_help3}
		\nabla_{\mb{w}} h^l(\mb{P}(\mb{w})) = \langle \nabla h^l(\mb{X}), \mb{P}^{\otimes l} \rangle_{1,2,4,5,\dots,3l-1,3l}
	\end{equation}
	Hence, if we take the derivate of $\nabla h^l$ and evaluate it at $\mb{X}$ in the direction of $\mb{U} \in \RR^{[n \times r] \circ l}$, we obtain that 
	\begin{align*}
		D_{\mb{X}} \nabla h^l(\mb{U}) &= 2\langle \nabla f^l( \langle \mb{X}, \mb{X} \rangle_{2*[l]}) ,\mb{U}\rangle_{2*[l]} + \langle \langle \mb{A}^{\otimes l}, \langle \mb{X}, \mb{U} \rangle_{2*[l]} + \langle \mb{U}, \mb{X} \rangle_{2*[l]} \rangle, \langle \mb{A}^{\otimes l}, \mb{w} \rangle_{2,5,\dots,3l-1} \rangle \\
		& + \langle \langle \mb{A}^{\otimes l}, \langle \mb{X}, \mb{U} \rangle_{2*[l]} + \langle \mb{U}, \mb{X} \rangle_{2*[l]} \rangle, \langle \mb{A}^{\otimes l}, \mb{w}  \rangle_{3,6,\dots,3l} \rangle
	\end{align*}
	Combined with \eqref{eq:lem_cp_highr_help3}, we conclude that
	\begin{align*}
		[\nabla^2_{\mb{w}} h^l(\mb{P}(\mb{w}))](\mb{v},\mb{v}) &= 2 \langle \nabla f^l( \mathcal{M}(\mb{w})) ,\mathcal{M}(\mb{v}) \rangle + \langle \langle \mb{A}^{\otimes l}, D_{\mb{w}} \mathcal{M}(\mb{v})\rangle, \langle \mb{A}^{\otimes l}, D_{\mb{w}} \mathcal{M}(\mb{v})   \rangle \rangle \\
		&= 2 \langle \nabla f^l(\mathcal{M}(\mb{w})) ,\mathcal{M}(\mb{v}) \rangle + \|\langle \mb{A}^{\otimes l}, D_{\mb{w}} \mathcal{M}(\mb{v})   \rangle\|^2_F
	\end{align*}
	which yields \eqref{eq:socp_lifted_highr} directly.
\end{proof}

Now, we turn to showcasing the relationship between the FOPs of \eqref{eq:lifted_main_highr} and those of \eqref{eq:unlifted_main}, which also have a one-to-one correspondence in the symmetric rank-1 regime. This is the reason why it is necessary to introduce \eqref{eq:lifted_main_highr} despite the extra complication, as rank-1 components tensors in $\RR^{[n \times r] \circ l}$ are not lifted versions of $X \in \RR^{n \times r}$.

\begin{theorem}\label{thm:focp_highr}
	For the lifted formulation \eqref{eq:lifted_main_highr}, the first-order condition $\nabla h^l(\mb{\hat w}) = 0$ holds for a symmetric rank-1 tensor $\mb{\hat w}$ if and only if
	\[
		\mb{\hat w} = \vecc(\hat X)^{\otimes l}
	\] 
	where $\hat X \in \RR^{n \times r}$ is an FOP of \eqref{eq:unlifted_main}.
\end{theorem}

\begin{proof}[Proof of Theorem \ref{thm:focp_highr}]
When $\mb{\hat w} = \vecc(\hat X)^{\otimes l}$, Lemma~\ref{lem:kron_iden} and \eqref{eq:focp_lifted_highr} together imply that
\begin{equation}\label{eq:thm_fop_highr_help1}
	\langle \nabla f^l( \langle \hat X^{\otimes l}, \hat X^{\otimes l} \rangle_{2*[l]}) , \hat X^{\otimes l} \rangle_{2*[l]} = (\nabla f(\hat X \hat X^\top) \hat X)^{\otimes l} = 0
\end{equation}
which is equivalent to 
\[
	\nabla f(\hat X \hat X^\top) \hat X = 0,
\]
which is exactly \eqref{eq:focp_unlifted}.
\end{proof}

Theorem~\ref{thm:focp_highr} establishes a robust connection between the first-order critical points of the lifted formulation and those of the unlifted formulation. This implies that when first-order methods approach a critical point in \eqref{eq:lifted_main_highr}, valuable information about an FOP of \eqref{eq:unlifted_main} can also be readily extracted. However, the primary challenge in optimizing \eqref{eq:unlifted_main} stems from spurious solutions, which cannot be escaped by first or even second-order algorithms. Consequently, it becomes crucial to examine whether the Hessians of the FOPs of \eqref{eq:lifted_main_highr}, especially those that correspond to the spurious solutions of \eqref{eq:unlifted_main}, exhibit any unique properties. As it turns out, the non-global FOPs of \eqref{eq:lifted_main_highr} display some highly favorable characteristics: they no longer constitute second-order critical points of \eqref{eq:lifted_main_highr} and are transformed into strict saddles when the parametrization level $l$ is sufficiently large.

To motivate our analysis of conversion from spurious solutions to strict saddle points, we first offer a closer analysis to the SOPs of the unlifted problem \eqref{eq:unlifted_main}, which also serves as the key intuition into our main results in this section.

The main observation is that, for a spurious SOP $\hat X$ and any ground truth $Z$ with $\hat X \hat X^\top \neq ZZ^\top$, although they all obey conditions \eqref{eq:focp_unlifted} and \eqref{eq:socp_unlifted}, they still have intrinsic differences that can be amplified via over-parametrization. To illustrate this phenomenon in more detail, we will introduce the following Lemma:
\begin{lemma}\label{lem:smallest_eig_socp}
	For an arbitrary FOP $\hat X \in \RR^{n \times r}$ of \eqref{eq:unlifted_main} satisfying the $(\alpha_s,r)$-RSC property, the following inequality holds:
	\begin{equation}
		\lambda_{\min} (\nabla f(\hat X \hat X^\top )) \leq -\alpha_s \frac{\|\hat X \hat X^\top -M^*\|^2_F}{2\tr(M^*)} \leq 0
	\end{equation}
\end{lemma}
\begin{proof}[Proof for Lemma~\ref{lem:smallest_eig_socp}]

According to \cite{zhang2021general}, $\nabla f(M)$ can be assumed to be symmetric without loss of generality. Hence, one can select $u \in \RR^n$ such that $u^\top \nabla f(\hat x \hat x^\top)u = \lambda_{\text{min}}(\nabla f(\hat x \hat x^\top))$. Then via the definition of RSC we have
	\[
		f(M^*) \geq f(\hat X \hat X^\top) + \langle \nabla f(\hat X \hat X^\top ), M^* - \hat X \hat X^\top \rangle + \frac{\alpha_s}{2}  \|\hat X \hat X^\top -M^*\|^2_F.
	\]
	Given that $\hat X$ is also an FOP, we have that
	\[
		\langle \nabla f(\hat X \hat X^\top ), \hat X \hat X^\top \rangle = 0
	\]
	according to \eqref{eq:focp_unlifted} and since $f(\hat X \hat X^\top) - f(M^*) \geq 0$, one can write that
	\[
		\langle \nabla f(\hat X \hat X^\top ), M^* \rangle \leq -\frac{\alpha_s}{2} \|\hat x \hat x^\top -M^*\|^2_F
	\]
	after rearrangements. Furthermore, since both $\nabla f(\hat X \hat X^\top )$ and $M^*$ are assumed to be positive semidefinite for the above-mentioned reasons, we have that
	\[
		\langle \nabla f(\hat X \hat X^\top ), M^* \rangle \geq \lambda_{\min} (\nabla f(\hat X \hat X^\top )) \tr(M^*)
	\]
	which implies that 
	\begin{equation}\label{eq:G_bound}
		\lambda_{\min} (\nabla f(\hat X \hat X^\top )) \leq -\alpha_s \frac{\|\hat X \hat X^\top -M^*\|^2_F}{2\tr(M^*)} \leq 0
	\end{equation}
	This completes the proof.
\end{proof}

Now let us recall \eqref{eq:socp_unlifted}, which can be stated equivalently as  
\[
	\lambda_{\min} (\nabla f(\hat X \hat X^\top )) \geq - [\nabla^2 f(\hat X \hat X^\top)](\hat X U^\top + U \hat X^\top,\hat X U^\top + U \hat X^\top) \quad \forall U
\]
By using the $(L_s,r)$-RSS property and the assumption that the sensing matrices are symmetric, we can further lower-bound the right-hand side of the above inequality as 
\[
	- [\nabla^2 f(\hat X \hat X^\top)](\hat X U^\top + U \hat X^\top,\hat X U^\top + U \hat X^\top) \geq - 4 [\nabla^2 f(\hat X \hat X^\top)](\hat XU^\top) \geq -4 L_s \| \hat XU^\top \|^2_F
\]

Therefore, it is easy to see that a sufficient condition for the spurious SOPs to disappear is
\begin{equation}\label{eq:no_spurious_ineq}
	\alpha_s \frac{\|\hat X \hat X^\top -M^*\|^2_F}{2\tr(M^*)} \geq 4 L_s \| \hat XU^\top \|^2_F \quad \forall U
\end{equation}
which means that the $L_s$ and $\alpha_s$ parameters should be benign, and this essentially constitutes the main proof strategy in the existing literature showing in-existence of spurious solutions under benign RIP or RSS/RSC conditions \cite{zhang2019sharp,zhang2021general,zhang2020many,ma2022sharp,ma2023noisy}.

Therefore, it is natural to ask, in the case when $L_s$ and $\alpha_s$ do not satisfy \eqref{eq:no_spurious_ineq}, whether one can systematically over-parametrize the problem so that the LHS of \eqref{eq:no_spurious_ineq} eventually becomes bigger than the RHS. We know that if we just raise both the RHS and LHS to arbitrary powers, the sign of the inequality will not flip. Therefore, the key insight is that if we keep the constant 4 unchanged, and lift the other terms to arbitrary powers, we can eventually satisfy \eqref{eq:socp_unlifted}. In general terms, we take the following steps in order to establish a strong result regarding the conversion of spurious solutions to strict saddle points:
\begin{enumerate}
	\item Proving that $\langle \nabla f^l(\langle \mb{P}(\mb{\hat w}), \mb{P}(\mb{\hat w}) \rangle, \Delta \otimes \Delta \rangle \geq |\lambda_{\min} (\nabla f(\hat X \hat X^\top ))|^l$ for some appropriately chosen point $\Delta \in \RR^{nr \circ l}$.
	\item Proving that $\| \langle \mb{A}^{\otimes l} , \langle \mb{P}(\mb{w}), \mb{P}(\Delta) \rangle_{2*[l]} + \langle \mb{P}(\Delta), \mb{P}(\mb{w}) \rangle_{2*[l]} \rangle \|_F^2 \leq 4 L_s \| \hat X U^\top \|^{2l}_F$ for some appropriately chosen points $\Delta \in \RR^{nr \circ l}$ and $U \in \RR^{n \times r}$
	\item Finding the smallest $l$ that converts the spurious solution to strict saddle point, under mild technical conditions.
\end{enumerate}

Now we turn to the main result of the general-rank scenario, which concerns the conversion of spurious solutions to strict saddle points. We present the formal results below.

\begin{theorem}\label{thm:socp_highr}
	Consider an SOP $\hat X \in \RR^{n \times r}$ of \eqref{eq:unlifted_main} of general rank $r < n$, such that $\hat X \hat X^\top \neq M^*$, and assume that \eqref{eq:unlifted_main} satisfies the RSC and RSS conditions. Then $\mb{\hat w} = \vecc(\hat X)^{\otimes l}$ is a strict saddle of \eqref{eq:lifted_main_highr} with a rank-1 symmetric escape direction if $\hat X$ satisfies the inequality
	\begin{equation}\label{eq:thm_socp_distance_lower_highr}
		\|M^* - \hat X \hat X^\top \|^2_F \geq \frac{L_s}{\alpha_s} \lambda_r(\hat X \hat X^\top) \tr(M^*)
	\end{equation}
	and $l$ is odd and is large enough so that
	\begin{equation}\label{eq:thm_socp_lcond_highr}
		l > \frac{1}{1-\log_2(2 \beta)} 
	\end{equation}
	where $\beta$ is defined as
	\[
		\beta \coloneqq \frac{ L_s \tr(M^*) \lambda_r(\hat X \hat X^\top)}{\alpha_s \|M^* - \hat X \hat X^\top \|^2_F}.
	\]
	Here, $L_s$ and $\alpha_s$ are the respective RSS and RSC constants of \eqref{eq:unlifted_main}.
\end{theorem}
\begin{proof}[Proof of Theorem~\ref{thm:socp_highr}]
	By Lemma~\ref{lem:smallest_eig_socp}, we select  $u \in \RR^n$ such that $u^\top \nabla f(\hat X \hat X^\top)u = \lambda_{\text{min}}(\nabla f(\hat X \hat X^\top))$ with $\lambda_{\text{min}}(\nabla f(\hat X \hat X^\top)) \leq 0$ .
	
	Now define $G \coloneqq - \lambda_{\text{min}}(\nabla f(\hat X \hat X^\top)) \geq 0$. If we label 
	\[
		C_1 \coloneqq \langle \nabla f(\hat X \hat X^\top), U U^\top \rangle, \quad C_2 \coloneqq [\nabla^2 f(\hat X \hat X^\top)](\hat X U^\top,\hat X U^\top)
	\]
	Then we have that $C_1 = -G$. Also, since the sensing matrices $A_a$ can be assumed be to symmetric, we have that 
	\[
		[\nabla^2 f(\hat X \hat X^\top)](\hat X U^\top + U \hat X^\top,\hat X U^\top + U \hat X^\top) = 4 [\nabla^2 f(\hat X \hat X^\top)](\hat X U^\top,\hat X U^\top).
	\]
	Additionally we choose $q \in \RR^{r}$ to be the $r$-th singular value of $\hat X$, with
	\[
		\|\hat X q\|_2 = \sigma_r(\hat X), \qquad \|q\|_2=1
	\]
	and define $U \in \RR^{n \times r} = uq^\top$. Subsequently, the RSS condition can be used to show that
	\begin{equation*}
		\begin{aligned}
			 &[\nabla^2 f(\hat X \hat X^\top)](\hat X U^\top + U \hat X^\top,\hat X U^\top + U \hat X^\top) \leq L_s \|\hat X U^\top + U \hat X^\top\|^2_F  \\
			&= L_s \|u(\hat X q)^\top + (\hat X q)u^\top \|^2_F = 2 L_s \|\hat X q\|^2_F + 2 L_s (q^\top (\hat X^\top u))^2 = 2 L_s \lambda_r(\hat X \hat X^\top)
		\end{aligned}
	\end{equation*}
	since $\hat X^\top u = 0$ according to the first-order condition \eqref{eq:focp_unlifted}. Therefore,
	\[
		C_2 \leq \frac{1}{2} L_s \lambda_r(\hat X \hat X^\top)
	\]
	
	Now, if we choose $\Delta = \vecc(U)^{\otimes l}$ for the aforementioned $U \in \RR^{n \times r}$, the LHS of \eqref{eq:socp_lifted_highr} can be expressed as:
	\begin{equation}\label{eq:thm_socp_highr_help1}
		\begin{aligned}
			\text{LHS} &= 2 ( \langle \mb{A}, \hat X \hat X^\top \rangle_{2,3}^\top \langle \mb{A}, uu^\top \rangle_{2,3})^{l} - 2 ( \langle \mb{A}, M^* \rangle_{2,3}^\top \langle \mb{A}, uu^\top \rangle_{2,3})^{l} + 4 ( \|\langle \mb{A}, \hat X U^\top \rangle_{2,3}\|^2_2)^{l} \\
			&\leq  2 (\lambda_{\text{min}}(\nabla f(\hat X \hat X^\top)))^l + 4 C_2^l \\
			&=  2 C_1^l + 4 C_2^l
		\end{aligned}
	\end{equation}
	where the inequality follows from:
	\[
		a^n - b^n \leq (a-b)^n, \quad \forall b \geq a \geq 0
	\]
	Here, since $a-b = C_1 \leq 0$, the above inequality can be used. As a result,
	\[
		\text{LHS of \eqref{eq:socp_lifted_highr}} \leq \underbrace{-2G^l}_{\text{Part 1}} + \underbrace{\frac{2}{2^{l-1}} L^l_s \lambda_r(\hat X \hat X^\top)^l}_{\text{Part 2}}
	\]
	We know since $G \geq 0$, Part 1 is always negative assuming $l$ is odd, and Part 2 is always positive. Therefore, it suffices to find an order $l$ such that 
	\begin{equation}\label{eq:l_condition_highr}
		G^l > (1/2^{l-1}) L^l_s \lambda_r(\hat X \hat X^\top)^l
	\end{equation}
	To derive a sufficient condition for \eqref{eq:l_condition_highr}, we first need a lower bound on $G$, and Lemma \eqref{lem:smallest_eig_socp} conveniently provides this bound, giving that
	\begin{equation}\label{eq:G_lower_highr}
		G \geq \frac{\alpha_s}{2 \tr(M^*)} \|M^* - \hat X \hat X^\top \|^2_F
	\end{equation}
	Therefore, if
	\[
		\left( \frac{\alpha_s}{2 \tr(M^*)} \|M^* - \hat X \hat X^\top \|^2_F \right)^l > (1/2^{l-1}) L^l_s \lambda_r(\hat X \hat X^\top)^l,
	\]
	we can conclude that \eqref{eq:l_condition_highr} holds, which implies that the LHS of \eqref{eq:socp_lifted_highr} is negative, directly proving that $\hat X^{\otimes l}$ is not an SOP anymore. Elementary manipulations of the above equation give that a sufficient condition is 
	\begin{equation}\label{eq:l_condition2}
		\|M^* - \hat X \hat X^\top \|^2_F > 2^{1/l} \frac{L_s}{\alpha_s} \lambda_r(\hat X \hat X^\top) \tr(M^*)
	\end{equation}
	We now consider \eqref{eq:thm_socp_distance_lower_highr}, which means that
	\begin{equation}\label{eq:xhat_upper_D_highr}
		\lambda_r(\hat X \hat X^\top)  \leq \frac{\alpha_s}{L_s \tr(M^*)} \|M^* - \hat X \hat X^\top \|^2_F
	\end{equation}
	Subsequently, define a constant $\gamma$ such that
	\[
		L_s \lambda_r(\hat X \hat X^\top) = \gamma (\frac{\alpha_s}{2 \tr(M^*)} \|M^* - \hat X \hat X^\top \|^2_F) 
	\]
	Then, according to Lemma \ref{lem:G_upper} and \eqref{eq:G_lower_highr}, we can conclude that $\gamma \geq 1$. Moreover, \eqref{eq:xhat_upper_D_highr} also means that $\gamma < 2$. With this new definition, the sufficient condition \eqref{eq:l_condition2} becomes
	\begin{equation}\label{eq:l_condition3_highr}
		1 > \frac{\gamma}{2^{(l-1)/l}}
	\end{equation}
	Since we already know that $1 \leq \gamma < 2$, there always exists a large enough $l$ such that \eqref{eq:l_condition3_highr} holds, which in turn implies that LHS of \eqref{eq:socp_lifted_highr} is negative, proving that $\vecc(\hat X)^{\otimes l}$ is a saddle point with the escape direction $\vecc(U)^{\otimes l}$, proving the claim.
	
	Next, we aim to study how large $l$ needs to be in order for \eqref{eq:l_condition3_highr} to hold. Again, we know that
	\[
		\gamma = \frac{2 L_s \tr(M^*) \lambda_r(\hat X \hat X^\top)}{\alpha_s \|M^* - \hat X \hat X^\top \|^2_F} \coloneqq 2 \beta
	\]
	and that $\beta \leq 1$ due to assumption \eqref{eq:thm_socp_distance_lower_highr}. Therefore, for \eqref{eq:l_condition3_highr} to hold true, it is enough to have 
	\[
		2^{(l-1)/l} > 2 \beta \implies \frac{l-1}{l} > \log_2(2 \beta) \implies l > \frac{1}{1-\log_2(2 \beta)} 
	\]
\end{proof}

\subsection{Other Considerations of Lifted Landscape}

In the previous sections, we have shown that by lifting the optimization problem \eqref{eq:unlifted_main} into tensor spaces, we could convert spurious local solutions into strict saddle points. However, it is also important that we could distinguish the true ground truth solutions $Z \in \RR^{n \times r}$ with $ZZ^\top. = M^*$ from the spurious ones. This requires that the true solutions $Z$ will remain SOPs after lifting, which we indeed prove in the following theorem:
\begin{theorem}\label{thm:socp_z}
	Assume that $Z \in \RR^{n \times r}$ is a ground truth solution of \eqref{eq:unlifted_main} such that $ZZ^\top = M^*$. Then $\vecc(Z)^{\otimes l}$ remains an SOP of \eqref{eq:lifted_main_highr} regardless of the parametrization level $l$, and without the need for \eqref{eq:unlifted_main} to satisfy the RSC or RSS conditions.
\end{theorem}
\begin{proof}[Proof of Theorem~\ref{thm:socp_z}]
	Let us start with the first-order optimality condition. Consider the linear map in the proof of Lemma~\ref{lem:cp_lifted_highr} $\mathcal{M}: \RR^{nr \circ l} \mapsto \RR^{n \circ 2l}$
	\begin{equation*}
		\mathcal{M}(\mb{w}) = \langle \mb{P}(\mb{w}), \mb{P}(\mb{w}) \rangle_{2*[l]},
	\end{equation*}
	Again, it is apparent that
	\begin{equation*}
		\nabla f^l( \mb{M}) =  \langle \langle \mb{A}^{\otimes l}, \mb{M} - \mathcal{M}(\vecc(Z)^{\otimes l})\rangle, \mb{A}^{\otimes l} \rangle_{1,4,\dots,3l-2}
	\end{equation*}
	Therefore, at the point $\mb{M}$ = $\mathcal{M}(\vecc(Z)^{\otimes l})$, we know that $\nabla f^l(\mathcal{M}(\vecc(Z)^{\otimes l})) = 0$. Consequently, the LHS of \eqref{eq:focp_lifted_highr} is equal to zero since it is a product between $\nabla f^l(\mathcal{M}(\vecc(Z)^{\otimes l}))$ and $\mb{P}(\vecc(Z)^{\otimes l})$.
	
	Next, we turn to the second-order optimality condition. Again, recall from the proof of Lemma~\ref{lem:cp_lifted_highr} that
	\[
	 	\text{LHS of \eqref{eq:socp_lifted_highr}} 	= \underbrace{2 \langle \nabla f^l(\mathcal{M}(\mb{w})) ,\mathcal{M}(\mb{\Delta}) \rangle}_{\text{Part 1}} + \underbrace{\|\langle \mb{A}^{\otimes l}, D_{\mb{w}} \mathcal{M}(\mb{\Delta})   \rangle\|^2_F}_{\text{Part 2}}
	\]
	By the above arguments, we have $\nabla f^l(\mathcal{M}(\mb{w})) = 0$ when $\mb{w} = \vecc(Z)^{\otimes l}$, meaning that Part 1 equals to zero. This implies that
	\[
		\text{LHS of \eqref{eq:socp_lifted_highr}} 	= \|\langle \mb{A}^{\otimes l}, D_{\mb{w}} \mathcal{M}(\mb{\Delta})   \rangle\|^2_F \geq 0, \qquad \forall \mb{\Delta}
	\]
	regardless of the values of $\mb{A}$ or $\mb{w} = \vecc(Z)^{\otimes l}$. 
\end{proof}

Next, it is important to analyze the main results obtained in Theorem~\ref{thm:socp_highr} under the lens of some other existing characterizations of the loss landscape of \eqref{eq:unlifted_main}. According to Theorem~\ref{thm:socp_highr}, over-parametrization or lifting proves to be highly beneficial when dealing with spurious solutions, represented by $\hat X$, that significantly deviate from the actual ground truth. The theorem implies that as the distance between $\hat X$ and the ground truth increases, a smaller value of $l$ is necessary for $\vecc(\hat X)^{\otimes l}$ to evolve into a saddle point, as alluded to in \eqref{eq:thm_socp_lcond_highr}. This concept is consistent with previous studies, which maintain that the area surrounding $M^*$ exhibits a favorable optimization landscape, characterized by an absence of deceptive local solutions in a specified zone around $M^*$. A commonly cited illustration of this assertion is provided below.
\begin{theorem}[Theorem 3 \cite{bi2020global}]
	If $\hat X$ is an SOP of \eqref{eq:unlifted_main} and
	\begin{equation}\label{eq:thm_local_region_range}
		\|\hat X \hat X^\top - M^*\|_F \leq \frac{4L_s \alpha_s}{(L_s+\alpha_s)^2} \lambda_r(M^*),
	\end{equation}
	then
	\[
		\hat X \hat X^\top = M^*
	\]
\end{theorem}
This means that any spurious solution of \eqref{eq:unlifted_main} is reasonably far away from the ground truth solution $M^*$. Coupled with the fact that lifting the problem into higher-dimensional tensor spaces can convert spurious solutions far away from $M^*$ into strict saddles points, we can ascertain that by setting the RHS of \eqref{eq:thm_local_region_range} to be greater or equal to the RHS of \eqref{eq:thm_socp_distance_lower_highr}, all spurious solutions will be converted into strict saddle points via lifting. We make this key observation concrete in the following theorem.
\begin{theorem}\label{thm:socp_corollary}
	Assume that $\hat X \in \RR^{n \times r}$ is a spurious solution of \eqref{eq:unlifted_main}, and that \eqref{eq:unlifted_main} satisfies the RSC and RSS assumptions with the $\alpha_s$ and $L_s$ constants, respectively. Then $\vecc(\hat X)^{\otimes l}$ is a strict saddle point of \eqref{eq:lifted_main_highr} for an odd $l$ satisfying \eqref{eq:thm_socp_lcond_highr} if
	\begin{equation}
		\|M^*\|_F \leq \frac{1}{\tau \sqrt{r}}\frac{2\sqrt{2} \alpha_s^{5/2}}{(L_s+\alpha_s)^2 \sqrt{L_s}}
	\end{equation}
	where $\tau$ is the condition number of $M^*$.
\end{theorem}
\begin{proof}[Proof of Theorem~\ref{thm:socp_corollary}]
	By using Lemma~\ref{lem:xhat_upper}, we know that 
	\[
		\text{RHS of \eqref{eq:thm_socp_distance_lower_highr}} \leq \sqrt{\frac{2L_s^3}{r \alpha_s^3}} \|M^*\|_F \tr(M^*)
	\]
	Hence, it is enough to make the RHS of the above inequality to be less than that of \eqref{eq:thm_local_region_range}, meaning that
	\[
		\sqrt{\frac{2L_s^3}{r \alpha_s^3}} \|M^*\|_F \tr(M^*) \leq \frac{4L_s \alpha_s}{(L_s+\alpha_s)^2} \lambda_r(M^*) \implies \|M^*\|_F \frac{\tr(M^*)}{\lambda_r(M^*)} \leq \frac{2\sqrt{2r} \alpha_s^{5/2}}{(L_s+\alpha_s)^2 \sqrt{L_s}}
	\]
	Then, acknowledging that $\tr(M^*) \leq r \tau \lambda_r(M^*)$ completes the proof.
\end{proof}

\newpage
\section{Additional Details for Implicit Bias of GD in Tensor Space}\label{sec:app_implicit}

\subsection{More Tensor Algebra} \label{sec:app_implicit_tensor}
\begin{definition}\label{def:tensor_norm}
	Given a cubic tensor $\mb{w} \in \RR^{n \circ l}$, its spectral norm $\|\cdot\|_S$ and nuclear norm $\|\cdot\|_*$ are defined respectively as
	\begin{align*}
		\|\mb{w}\|_* &= \inf \left\{ \sum_{j=1}^{r_m} |\lambda_j|: \mb{w} = \sum_{j=1}^{r_m} \lambda_j w_j^{\otimes l}, \ \|w_j\|_2 = 1, w_j \in \RR^{n} \right\} \\
		\|\mb{w}\|_S &= \sup \left\{| \langle \mb{w}, u^{\otimes l} \rangle |\ \|u\|_2 = 1, u \in \RR^{n} \right\}
	\end{align*}
\end{definition}
From the definition, it also follows that 
\[
	\|\mb{w}\|_S \leq \|\mb{w}\|_*
\]
The above definitions are similar to those for their matrix counterparts. However, unlike the spectral norm of matrices, the spectral norm of tensors are not tensor norms, namely that they do not obey
\[
	\|\langle \mb{w}, \mb{v} \rangle\|_S \leq \|\mb{w}\|_S \|\mb{v}\|_S
\]
in general. Conversely, the nuclear norm is a valid tensor norm, and we have the following property:
\begin{lemma}[Theorem 2.1, 3.2 \cite{qi2019tensor}]\label{lem:tensor_norm_upper}
	For tensors $\mb{w}$ and $\mb{v}$ of appropriate dimensions (if doing inner product, the dimensions along which the multiplication is performed must have matching size), we have
	\begin{align*}
		\|\langle \mb{w}, \mb{v} \rangle\|_S &\leq \|\mb{w}\|_S \|\mb{v}\|_* \\
		\|\langle \mb{w}, \mb{v} \rangle\|_* &\leq \|\mb{w}\|_* \|\mb{v}\|_*
	\end{align*}
\end{lemma}
Moreover, they have a dual norm relationship:
\begin{lemma}[Lemma 21 \cite{lim2013blind}] \label{lem:dual_tensor_norm}
	The spectral norm $\|\cdot\|_S$ is the dual norm to the nuclear norm $\|\cdot\|_*$, namely given an arbitrary tensor $\mb{w}$, we have that 
	\[
		\|\mb{w}\|_S = \sup_{\|\mb{v}\|_* \leq 1} | \langle \mb{w}, \mb{v} \rangle |  
	\] 
	with $\mb{v}$ having the same dimensions as $\mb{w}$.
\end{lemma}
Next, we introduce the notion of eigenvalues for tensors. There are many related definitions, like outlined in \cite{qi2012spectral}. However, we introduce a novel variational characterization of eigenvalues that resembles the Courant-Fisher minimax definition for eigenvalues of matrices. Note this is a new definition that is fist introduced in this paper, and may be of independent interest outside of the current scope.
\begin{definition*}[Definition 4, Variational Eigenvalue of Tensors]
	For a given tensor $\mb{w} \in \RR^{n \circ l}$, we define its $k^{th}$ variational eigenvalue (v-Eigenvalue) $\lambda_k^v(\mb{w})$ as
	\[
		\lambda_k^v(\mb{w}) \coloneqq \max_{\substack{S \\ \dim(S)=k}} \min_{\mb{u} \in S} \frac{|\langle \mb{w}, \mb{u} \rangle |}{\|\mb{u}\|^2_F}, \quad k \in [n]
	\]
	where $S$ is a subspace of $\RR^{n \circ l}$ that is spanned by a set of orthogonal, symmetric, rank-1 tensors. Its dimension denotes the number of orthogonal tensors that span this space. 
\end{definition*}
It is apparent from the definition that $\|\mb{w}\|_S = \lambda_1^v(\mb{w})$. Note that our definition of v-Eigenvalues of tensors can only define $n$ eigenvalues at most, which is not the maximum amount of H- or Z-Eigenvalues a tensor can have \cite{qi2012spectral}, and it is well known that even with symmetric tensors, its rank can go well beyond $n$  \cite{comon2008symmetric}. We also note that this definition exactly coincides with the definition of Hermitian tensor eigenvalues (introduced here \cite{ni2019hermitian}) when constrained to Hermitian tensors \cite{chang2021hanson}. We also conjecture that this definition coincides with the top-n Z-Eigenvalues for even-order symmetric real tensors \cite{qi2012spectral}, but it is an open question for now.

Using the definition of v-Eigenvalues, we can also obtain an equivalent characterization, just like the Courant-Fisher definition for matrix eigenvalues, which helps us in proving a tensor version of Weyl's inequality:
\begin{proposition}\label{prop:eigen_equi}
	For an integer $k$ in $[1,\dots,n]$, the $k^{th}$ variational eigenvalue (v-Eigenvalue) $\lambda_k^v(\mb{w})$ of a tensor $\mb{w}$ satisfies:
	\[
		\lambda_k^v(\mb{w}) = \min_{\substack{T \\ \dim(T)=n-k+1}} \max_{\mb{u} \in T} \frac{|\langle \mb{w}, \mb{u} \rangle |}{\|\mb{u}\|^2_F} = \max_{\substack{S \\ \dim(S)=k}} \min_{\mb{u} \in S} \frac{|\langle \mb{w}, \mb{u} \rangle |}{\|\mb{u}\|^2_F}
	\]

\end{proposition}
\begin{proof}[Proof of Proposition~\ref{prop:eigen_equi}]
	We prove the proposition by contradiction. Assume that the two formulations claimed to be identical in Proposition~\ref{prop:eigen_equi} are not the same. We further assume that $S$ is spanned by symmetric, rank-1 tensors $\{ \mb{u}_1, \dots, \mb{u}_k \}$, and that $T$ is spanned by symmetric, rank-1 tensors $\{ \mb{u}_{-(n-k+1)}, \dots, \mb{u}_{-1} \}$, meaning that 
	\[
		\langle \mb{w}, \mb{u}_k \rangle \neq  \langle \mb{w}, \mb{u}_{-(n-k+1)} \rangle
	\]
	assuming that $\mb{u}_k$ and $\mb{u}_{-(n-k+1)}$ are the \emph{inner} argmin and argmax of their respective formulations with norm 1. Since they have to be rank-1 tensors (if not we can decrease the proportion of orthogonal elements with higher or lower $|\langle \mb{w}, \mb{u} \rangle |$ values), it is possible to denote 
	\[
		\mb{u}_k = u_k^{\otimes l}, \quad \mb{u}_{-(n-k+1)} = u_{-(n-k+1)}^{\otimes l} \ \text{where}\  u_k, u_{-(n-k+1)} \in \RR^n
	\]
	We also know that $u_k$ and $u_{-(n-k+1)}$ are linearly independent, as otherwise $\mb{u}_k$ and $\mb{u}_{-(n-k+1)}$ will have the same inner product with $\mb{w}$. Thus, assume 
	\[
		u_k = \xi_1 u_{-(n-k+1)} + \xi_2 u_{-(n-k+1)}^\perp, \quad \xi_2 \neq 0.
	\]
	It follows that 
	\[
		\mb{u}_k = \xi_1^l u_{-(n-k+1)}^{\otimes l} + \xi_2^l (u_{-(n-k+1)}^\perp)^{\otimes l} + \underbrace{\dots \dots}_{\text{other non-symmetric terms}}
	\]
	Denote $(u_{-(n-k+1)}^\perp)^{\otimes l} \coloneqq \mb{u}_{k+1}$. Now, it follows from definition that 
	\[
		\mb{u}_{k+1} \perp \{ \mb{u}_1, \dots, \mb{u}_{k-1} \}
	\]
	and also 
	\[
		\mb{u}_{k+1} \notin \text{span}\{ \mb{u}_{-(n-k)}, \dots, \mb{u}_{-1} \}
	\]
	as otherwise the \emph{outer} maximization formulation affecting the choice of $u_k$ will make $\xi_2 = 0$, contradicting our claim. By definition we have $\text{span}\{ \mb{u}_{1}, \dots, \mb{u}_{k} \} \bigcap \text{span}\{ \mb{u}_{-(n-k)}, \dots, \mb{u}_{-1} \}= \{ \emptyset \}$.
	
	In summary we have that $\mb{u}_{k+1} \perp \mb{u}_{-(n-k+1)}, \{ \mb{u}_1, \dots, \mb{u}_{k-1} \}, \{ \mb{u}_{-(n-k)}, \dots, \mb{u}_{-1}  \}$, meaning that we have obtained $n+1$ symmetric rank-1 and $n$-dimensional tensors all orthogonal to each other, which is apparently not possible, thus refuting our initial claim.
	\end{proof}
	With this new definition equipped, we proceed to show a tensor version of Weyl's inequality, which is key in our proof as promised.
\begin{lemma}[Tensor Weyl's]\label{lem:tensor_weyl}
	Consider two tensors $\mb{w}$ and $\mb{v}$ of the  same dimension. It holds that 
	\begin{equation}
		\lambda_k^v(\mb{w}) + \lambda_1^v(\mb{v}) \geq \lambda_k^v(\mb{w}+\mb{v}) \geq \lambda_k^v(\mb{w})- \lambda_1^v(\mb{v})
	\end{equation} 
\end{lemma}
The proof of Lemma~\ref{lem:tensor_weyl} is highly similar to that of Theorem 2 in \cite{chang2021hanson}, only substituting for our new definition of v-Eigenvalues, thus omitted for simplicity.

\subsection{Main Results and Their Proofs} \label{sec:app_implicit_proofs}
\emph{Note that in this section some tensor inner products will be written as if they were matrices for clarity of writing, and some subscripts for inner-products will be dropped when obvious. If two tensors in $\RR^{nr \circ 2l}$ are multiplied together, then the even dimensions of the first tensor will be inner-producted with the odd dimensions of the second tensor. When a tensor in $\RR^{nr \circ 2l}$ multiplies with a tensor in $\RR^{nr \circ l}$, then the even dimensions of the first tensor will be inner-producted with all the dimensions of the second tensor.}

We start with the proof to Lemma~\ref{lem:gd_sym}.
\begin{proof}[Proof of Lemma~\ref{lem:gd_sym}]
	We proceed with the proof by induction. First, assume that $\mb{w}_0  = x_0^{\otimes l}$ for some $x_0 \in \RR^{nr}$. One can write
	\begin{equation}\label{eq:highr_grad_exp_vec}
		\nabla h^l(\mb{w}_0)  = \langle \langle (I_r \oslash_{1,2} \mb{A})^{\otimes l}, \mb{w}_0 \rangle_{2*[l]}, \langle \mb{A}^{\otimes l}, \mathcal{M}(\mb{w}_0) - \mathcal{M}(\vecc(Z)^{\otimes l})\rangle \rangle_{1,3,\dots,2l-1}
	\end{equation}
	where $\mathcal{M}(\cdot)$ is defined per proof of Lemma~\ref{lem:cp_lifted_highr}. The difference between this formulation and \eqref{eq:highr_grad_exp} is that we have replaced $\langle  \mb{A}^{\otimes l}, \mb{P}(\mb{w}_0) \rangle_{2*[l]}$ with $\langle (I_r \oslash_{1,2} \mb{A})^{\otimes l}, \mb{w}_0 \rangle_{2*[l]}$, which are equivalent, just with the second tensor having the dimensions $nr,m,\dots,nr,m$ so that $\nabla h^l(\mb{w}_0)$ has the dimensions $nr, \dots , nr$. Note that $\oslash$ denotes the usual  kronecker product, which can be thought of a reshaped version of tensor outer product. $\oslash_{1,2}$ denotes the kronecker product only happening with respect to the first 2 dimensions of $\mb{A}$. From now on, we denote $\mb{A}_r \coloneqq I_r \oslash_{1,2} \mb{A}$.
	
	Now, according to the above formulation and Lemma~\ref{lem:kron_iden}, we have
	\begin{equation}\label{eq:h_l_grad_express}
		\begin{aligned}
			\nabla h^l(\mb{w}_0) &= \left( \langle \mb{A}_r, \langle \mb{A}, \mat(x_0) \mat(x_0)^\top - M^* \rangle \rangle_{3,6,\dots,3l} \ x_0 \right)^{\otimes l}\\
		&\coloneqq (\langle  \mb{A}_r^* \mb{A}, \mat(x_0) \mat(x_0)^\top - M^* \rangle \ x_0)^{\otimes l}
		\end{aligned}
	\end{equation}
	where 
	\begin{equation}\label{eq:A_r_def}
		(\mb{A}^l_r)^* \mb{A}^l \coloneqq \langle (\mb{A}_r)^{\otimes l}, \mb{A}^{\otimes l} \rangle_{3,6,\dots,3l} \in \RR^{[nr \times nr \times n \times n] \circ l}
	\end{equation}
	Now, $\langle  \mb{A}_r^* \mb{A}, \mat(x_0) \mat(x_0)^\top - M^* \rangle$ is an $nr \times nr$ matrix, so the above tensor is simply a vector outer product, being symmetric by definition. Consequently, $\mb{w}_1 = \mb{w}_0 - \eta \nabla h^l(\mb{w}_0)$ is still symmetric, since the addition of symmetric tensors maintains symmetric property. This completes the proof of the initial step.
	
	Then, we proceed to show the induction step. Assume that  $\mb{w}_{t-1}$ is symmetric, meaning that
	\[
		\mb{w}_{t-1} = \sum_{j=1}^{r_m} \lambda_j (x^{t-1}_j)^{\otimes l}, \quad x^{t-1}_j \in \RR^{nr}
	\]
	where $r_m$ is the symmetric rank of $\mb{w}_{t-1}$. This means that 
	\begin{align*}
		\nabla h^l(\mb{w}_{t-1}) =   &\sum_{j_1,j_2,j_3}^{r_m ,r_m, r_m} \lambda_{j_1} \lambda_{j_2} \lambda_{j_3} (\langle  \mb{A}_r^* \mb{A}, \mat(x^{t-1}_{j_1}) \mat(x^{t-1}_{j_2})^\top \rangle x^{t-1}_{j_3})^{\otimes l}  - \\ &\sum_{j_3}^{r_m } \lambda_{j_3} (\langle  \mb{A}_r^* \mb{A}, M^* \rangle x^{t-1}_{j_3} )^{\otimes l}
	\end{align*}
	which again is a weighted sum of rank-1 symmetric tensors, thus being symmetric. This shows that $\mb{w}_t = \mb{w}_{t-1} - \eta \nabla h^l(\mb{w}_{t-1})$ is also symmetric, concluding the induction step, thereby proving the claim.
\end{proof}

Next, we show the breakdown of tensors along the GD trajectory
\begin{lemma}\label{lem:traj_breakdown}
	The GD trajectory of \eqref{eq:lifted_main_highr} $\{ \mb{w}_{t}\}_{t=0}^{\infty}$ admits the following breakdown for an arbitrary $t$:
	\begin{equation}\label{eq:wt_breakdown}
		\mb{w}_{t+1} = \langle \mb{Z}_t, \mb{w}_0 \rangle - \mb{E}_t \coloneqq \mb{\tilde w}_t - \mb{E}_t
	\end{equation}
	where 
	\begin{align*}
		\mb{Z}_t &\coloneqq  (\mathcal{I}+ \eta \langle (\mb{A}^l_r)^* \mb{A}^l, (M^*)^{\otimes l} \rangle)^t \\
		\mb{E}_t &\coloneqq \sum_{i=1}^t (\mathcal{I}+ \eta \langle (\mb{A}^l_r)^* \mb{A}^l, (M^*)^{\otimes l} \rangle)^{t-i} \mb{\hat E}_i \\
		 \mb{\hat E}_i &\coloneqq \eta \langle \langle (\mb{A}^l_r)^* \mb{A}^l, \langle \mb{P}(\mb{w}_{i-1}), \mb{P}(\mb{w}_{i-1})\rangle_{2*[l]} \rangle, \mb{w}_{i-1} \rangle_{2*[l]}
	\end{align*}
	and where $(\mb{A}^l_r)^* \mb{A}^l \coloneqq \langle (\mb{A}_r)^{\otimes l}, \mb{A}^{\otimes l} \rangle_{3,6,\dots,3l} \in \RR^{[nr \times nr \times n \times n] \circ l}$.
\end{lemma}
\begin{proof}[Proof of Lemma~\ref{lem:traj_breakdown}]
	For this proof, we will proceed by induction. For $t=1$, we have that
	\begin{align*}
		\mb{w}_1 &= (\mathcal{I}+\eta \langle (\mb{A}^l_r)^* \mb{A}^l, (M^*)^{\otimes l} - \langle \mb{P}(\mb{w}_0), \mb{P}(\mb{w}_0)\rangle \rangle)\mb{w}_0 \\
		&= (\mathcal{I}+\eta \langle (\mb{A}^l_r)^* \mb{A}^l, (M^*)^{\otimes l} \rangle) \mb{w}_0 - \eta \langle (\mb{A}^l_r)^* \mb{A}^l, \langle \mb{P}(\mb{w}_0), \mb{P}(\mb{w}_0)\rangle \rangle \mb{w}_0 \\
		&= \langle \mb{Z}_1, \mb{w}_0 \rangle - \mb{E}_1
	\end{align*}
	Then, we move on to the induction step, while first assuming that it holds for some $t$. One can write
	\begin{align*}
		\mb{w}_{t+1} &= (\mathcal{I}+\eta \langle (\mb{A}^l_r)^* \mb{A}^l, (M^*)^{\otimes l} - \langle \mb{P}(\mb{w}_t), \mb{P}(\mb{w}_t)\rangle \rangle)\mb{w}_t \\
		&= (\mathcal{I}+\eta \langle (\mb{A}^l_r)^* \mb{A}^l, (M^*)^{\otimes l} \rangle) \mb{w}_t - \eta \langle (\mb{A}^l_r)^* \mb{A}^l, \langle \mb{P}(\mb{w}_t), \mb{P}(\mb{w}_t)\rangle \rangle \mb{w}_t  \\
		& = (\mathcal{I}+\eta \langle (\mb{A}^l_r)^* \mb{A}^l, (M^*)^{\otimes l} \rangle) \mb{w}_t - \mb{\hat E}_{t+1} \\
		&= (\mathcal{I}+\eta \langle (\mb{A}^l_r)^* \mb{A}^l, (M^*)^{\otimes l} \rangle) \left( \mb{\tilde w}_t - \sum_{i=1}^t (\mathcal{I}+ \eta \langle (\mb{A}^l_r)^* \mb{A}^l, (M^*)^{\otimes l} \rangle)^{t-i} \mb{\hat E}_i \right) - \mb{\hat E}_{t+1} \\
		&= \mb{\tilde w}_{t+1} - \sum_{i=1}^t (\mathcal{I}+ \eta \langle (\mb{A}^l_r)^* \mb{A}^l, (M^*)^{\otimes l} \rangle)^{t+1-i} \mb{\hat E}_i- \mb{\hat E}_{t+1} \\
		&= \mb{\tilde w}_{t+1} - \sum_{i=1}^{t+1} (\mathcal{I}+ \eta \langle (\mb{A}^l_r)^* \mb{A}^l, (M^*)^{\otimes l} \rangle)^{t+1-i} \mb{\hat E}_i \\
		&= \mb{\tilde w}_{t+1} - \mb{E}_t
	\end{align*}
\end{proof}

Following the second step in the main outline, we aim to bound the spectral norm of $\mb{E}_t$, via the next lemma.
\begin{lemma}\label{lem:E_t_spectral}
	Given a tensor $\mb{E}_t$ defined in Lemma~\ref{lem:traj_breakdown},  assume that $\mb{w}_0 = \epsilon x_0^{\otimes l}$, where $\epsilon \in \RR$ is the initialization scale. For every $t \leq t_s$,
	\begin{equation}\label{eq:E_t_spectral}
		\|\mb{E}_t\|_S \leq \frac{8}{r_U^l \sigma_1(U)^l} \epsilon^3 (n L_s)^{l/2} (1+\tilde \eta \sigma_1(U)^l)^{3t} \|x_0^{\otimes l}\|_*^3
	\end{equation}
	with 
	\begin{equation}\label{eq:ts_bound}
		t_s = \floor{\frac{\ln\left( \frac{\sigma_1^l(U)r_U^l}{8 r^l L_s^{l/2} \|x_0^{\otimes l}\|_*^3} \frac{|x_0^\top v_1|^l}{ n^{l/2}}\right)-2\ln(\epsilon)}{2\ln(1+\tilde \eta \sigma_1^l(U))}}
	\end{equation}
	where $U = \langle \mb{A}_r^* \mb{A},M^*\rangle \in \RR^{nr \times nr}$, $r_U$ being the rank of $U$, and $\tilde \eta = r_U^l \eta$. $\sigma_1(U)$ denotes the largest singular value of $U$, and $v_1$ being its associated singular vector. 
\end{lemma}
\begin{proof}[Proof of Lemma~\ref{lem:E_t_spectral}]
From Lemma~\ref{lem:tensor_norm_upper} and the definition in Lemma~\ref{lem:traj_breakdown}, it is apparent that 
	\begin{equation}\label{eq:et_norm_upper}
		\|\mb{E}_t\|_S \leq \sum_{i=1}^t \|(\mathcal{I}+ \eta \langle (\mb{A}^l_r)^* \mb{A}^l, (M^*)^{\otimes l} \rangle)^{t-i}\|_S \|\mb{\hat E}_i\|_*
	\end{equation}
	We proceed to derive upper bounds on the norm terms separately, and then combine them together later. We first deal with $\|\mb{\hat E}_{i}\|_*$. By Lemma~\ref{lem:tensor_norm_upper}, we have that
	\[
		\|\mb{\hat E}_i \|_* \leq \eta \|\langle (\mb{A}^l_r)^* \mb{A}^l, \langle \mb{P}(\mb{w}_{i-1}), \mb{P}(\mb{w}_{i-1})\rangle \rangle\|_* \|\mb{w}_{i-1}\|_*
	\]
	Now, assume that $\mb{w}_{i-1}$ admits the following breakdown
	\begin{equation}\label{eq:lem_E_help1}
		\mb{w}_{i-1} = \sum_{j=1}^{r_{i-1}} \lambda_j (x_j^{i-1})^{\otimes l}, \quad x_j^{i-1} \in \RR^{nr}, \ \|x_j^{i-1}\|_2=1
	\end{equation}
	where $\|\mb{w}_{i-1}\|_* = \sum_j |\lambda_j|$. Therefore,
	\[
		\langle \mb{P}(\mb{w}_{i-1}), \mb{P}(\mb{w}_{i-1})\rangle = \sum_{j_1,j_2}^{r_{i-1},r_{i-1}} \lambda_{j_1} \lambda_{j_2} \langle \mb{P}((x_{j_1}^{i-1})^{\otimes l} ), \mb{P}((x_{j_2}^{i-1})^{\otimes l}) \rangle,
	\]
	leading to 
	\[
		\langle (\mb{A}^l_r)^* \mb{A}^l, \langle \mb{P}(\mb{w}_{i-1}), \mb{P}(\mb{w}_{i-1})\rangle \rangle = \sum_{j_1,j_2}^{r_{i-1},r_{i-1}} \lambda_{j_1} \lambda_{j_2} \langle (\mb{A}^l_r)^* \mb{A}^l, \langle \mb{P}((x_{j_1}^{i-1})^{\otimes l}), \mb{P}((x_{j_2}^{i-1})^{\otimes l})\rangle \rangle.
	\]
	For given indices $j_1, j_2$ index, it follows from Lemma~\ref{lem:kron_iden} that
	\[
		\langle (\mb{A}^l_r)^* \mb{A}^l, \langle \mb{P}((x_{j_1}^{i-1})^{\otimes l}), \mb{P}((x_{j_2}^{i-1})^{\otimes l})\rangle \rangle = (\langle  \mb{A}_r^* \mb{A}, \mat(x_{j_1}^{i-1}) \mat(x_{j_2}^{i-1})^\top\rangle)^{\otimes l}
	\]
	Now, according to the definition of $\mb{A}_r \coloneqq I_r \oslash_{1,2} \mb{A}$, where $\oslash$ denotes the kronecker product (a reshaped tensor vector product, where the subscript denotes the dimension with which kronecker product is applied with respect to $\mb{A}$), we know that
	\[
		\langle \mb{A}_r^* \mb{A}, \mat(x_{j_1}^{i-1}) \mat(x_{j_2}^{i-1})^\top\rangle = I_r \oslash \langle \mb{A}^* \mb{A}, \mat(x_{j_1}^{i-1}) \mat(x_{j_2}^{i-1})^\top\rangle
	\]
	Hence, the eigenvalues of the LHS are just $r$ copies of that of the RHS \cite{petersen2008matrix}. This further implies
	\begin{align*}
		\|\langle \mb{A}_r^* \mb{A}, \mat(x_{j_1}^{i-1}) \mat(x_{j_2}^{i-1})^\top\rangle\|_* &= r\|\langle \mb{A}^* \mb{A}, \mat(x_{j_1}^{i-1}) \mat(x_{j_2}^{i-1})^\top\rangle\|_* \\
		&\leq r\sqrt{n} \|\langle \mb{A}^* \mb{A}, \mat(x_{j_1}^{i-1}) \mat(x_{j_2}^{i-1})^\top\rangle\|_F \\
		&\leq r\sqrt{n L_s} \|\mat(x_{j_1}^{i-1}) \mat(x_{j_2}^{i-1})^\top\|_F\\
		&= r\sqrt{n L_s}
	\end{align*}
	where the second last inequality follows from the RSS property, and the last equality follows from \eqref{eq:lem_E_help1}. Next, we apply Lemma~\ref{lem:tensor_norm_upper} again with
	\[
		\|\langle (\mb{A}^l_r)^* \mb{A}^l, \langle \mb{P}((x_{j_1}^{i-1})^{\otimes l}), \mb{P}((x_{j_2}^{i-1})^{\otimes l})\|_* \leq (\|\langle \mb{A}_r^* \mb{A}, \mat(x_{j_1}^{i-1}) \mat(x_{j_2}^{i-1})^\top\rangle\|_*)^l \leq r^l (n L_s)^{l/2}
	\]
	which leads to 
	\begin{align*}
		&\|\langle (\mb{A}^l_r)^* \mb{A}^l, \langle \mb{P}(\mb{w}_{i-1}), \mb{P}(\mb{w}_{i-1})\rangle \rangle\|_* \\
		\leq &\sum_{j_1,j_2}^{r_{i-1},r_{i-1}} |\lambda_{j_1}| |\lambda_{j_2}| \|\langle (\mb{A}^l_r)^* \mb{A}^l, \langle \mb{P}((x_{j_1}^{i-1})^{\otimes l}), \mb{P}((x_{j_2}^{i-1})^{\otimes l})\|_* \\
		\leq & r^l (n L_s)^{l/2} \sum_{j_1,j_2}^{r_{i-1},r_{i-1}} |\lambda_{j_1}| |\lambda_{j_2}|= r^l (n L_s)^{l/2} \|\mb{w}_{i-1}\|^2_* 
	\end{align*}
	This directly gives
	\[
		\|\mb{\hat E}_i \|_* \leq \eta (r^2 n L_s)^{l/2} \|\mb{w}_{i-1}\|^3_*
	\]
	Since our goal is to bound $\|\mb{E}_t\|_S$, we focus on  $\|(\mathcal{I}+ \eta \langle (\mb{A}^l_r)^* \mb{A}^l, (M^*)^{\otimes l} \rangle)^{t-i}\|_S$. Using binomial formula, we obtain that
	\[
		(\mathcal{I}+ \eta \langle (\mb{A}^l_r)^* \mb{A}^l, (M^*)^{\otimes l} \rangle)^{t-i} = \sum_{k=0}^{t-i} \binom{t-i}{k} \eta^k (\langle (\mb{A}^l_r)^* \mb{A}^l, (M^*)^{\otimes l} \rangle)^k
	\]
	where $\langle (\mb{A}^l_r)^* \mb{A}^l, (M^*)^{\otimes l} \rangle \in \RR^{nr \circ 2l}$, and $(\cdot)^k$ just denotes repeated multiplications along the even dimensions of the tensor, as explained in the disclaimer. To upper-bound the spectral norm of $(\mathcal{I}+ \eta \langle (\mb{A}^l_r)^* \mb{A}^l, (M^*)^{\otimes l} \rangle)^{t-i}$, it is necessary to upper-bound the spectral norm of $(\langle (\mb{A}^l_r)^* \mb{A}^l, (M^*)^{\otimes l} \rangle)^k$. To do so, we use Lemma~\ref{lem:dual_tensor_norm} to reformulate
	\[
		\|\langle (\mb{A}^l_r)^* \mb{A}^l, (M^*)^{\otimes l} \rangle^k\|_S = \sup_{\|\mb{v}\|_* \leq 1} | \langle (\mb{A}^l_r)^* \mb{A}^l, (M^*)^{\otimes l} \rangle^k, \mb{v} \rangle |
	\]
	Assume that the above supremum is achieved at $\mb{v}^*$, with nuclear norm decomposition of
	\[
		\mb{v}^* = \sum_{j_v=1}^{r_v} \lambda_{j_v} x_{j_v,1} \otimes \dots \otimes x_{j_v,2l}, \quad x_{j_v,p} \in \RR^{nr}, \ \|x_{j_v,p}\|_2=1 \ \forall p \in [2l]
	\]
	with $\sum_{j_v} |\lambda_{j_v}| = \|\mb{v}^*\|_* \leq 1$. Note that this decomposition is due to the fact that $\mb{v}$ is not necessarily symmetric. Again, by Lemma~\ref{lem:kron_iden}, 
	\[
		\langle (\mb{A}^l_r)^* \mb{A}^l, (M^*)^{\otimes l} \rangle^k = \left [ (\langle  \mb{A}_r^* \mb{A}, M^* \rangle)^k \right]^{\otimes l}, 
	\]
	directly leading to
	\[
		\|\langle (\mb{A}^l_r)^* \mb{A}^l, (M^*)^{\otimes l} \rangle^k\|_S = \sum_{j_v=1}^{r_v} | \lambda_{j_v} \prod_{p=0}^{l-1} x_{j_v,p*2}^\top \langle  \mb{A}_r^* \mb{A}, M^* \rangle^k  x_{j_v,p*2+1} |
	\]
	Since 
	\[
		x_{j_v,p*2}^\top \langle  \mb{A}_r^* \mb{A}, M^* \rangle^k x_{j_v,p*2+1} \leq \sigma_1^k(U)
	\]
	this means that 
	\[
		\|\langle (\mb{A}^l_r)^* \mb{A}^l, (M^*)^{\otimes l} \rangle^k\|_S = (\sigma_1^k(U))^l \sum_{j_v=1}^{r_v} | \lambda_{j_v}| \leq \sigma_1^{kl}(U)
	\]
	Going back to $(\mathcal{I}+ \eta \langle (\mb{A}^l_r)^* \mb{A}^l, (M^*)^{\otimes l} \rangle)^{t-i}$, 
	\begin{align*}
		\|(\mathcal{I}+ \eta \langle (\mb{A}^l_r)^* \mb{A}^l, (M^*)^{\otimes l} \rangle)^{t-i}\|_S &\leq \sum_{k=0}^{t-i} \binom{t-i}{k} \eta^k \|(\langle (\mb{A}^l_r)^* \mb{A}^l, (M^*)^{\otimes l} \rangle)^k\|_S \\
		&\leq \sum_{k=0}^{t-i} \binom{t-i}{k} \eta^k \sigma_1^{kl}(U) = (1+\eta \sigma_1^l(U))^{t-i}.
	\end{align*}
	Before further upper-bounding $\|\mb{E}_t\|_S$, we define $t_s$ in such a way that
	\begin{equation} \label{eq:ts_condition}
		\|\mb{\tilde w}_{t} - \mb{w}_t\|_* \leq \|\mb{\tilde w}_{t} \|_*, \quad \forall t \leq t_s
	\end{equation}
	where $\mb{\tilde w}_{t}$ is defined in \eqref{eq:wt_breakdown}. We will later justify the existence of $t_s$ and derive a lower bound. If the above inequality holds true, we also have
	\[
		 \|\mb{w}_{t} \|_* \leq  \|\mb{\tilde w}_{t} \|_* + \|\mb{\tilde w}_{t} - \mb{w}_t\|_* \leq 2 \|\mb{\tilde w}_{t} \|_*.
	\]
	Recall the binomial formula again and decompose $\mb{\tilde w}_{t}$ into
	\begin{equation} \label{eq:tildew_breakdown}
		\mb{\tilde w}_{t} = \sum_{k=0}^t \binom{t}{k} \eta^k \langle (\mb{A}^l_r)^* \mb{A}^l, (M^*)^{\otimes l} \rangle^k \mb{w}_0
	\end{equation}
	Therefore, it follows from  Lemma~\ref{lem:tensor_norm_upper} that,
	\begin{equation}\label{eq:tildew_nuclear_upper}
		\|\mb{\tilde w}_{i-1} \|_* \leq \left( \sum_{k=0}^{i-1} \binom{i-1}{k} \eta^k \| \langle (\mb{A}^l_r)^* \mb{A}^l, (M^*)^{\otimes l} \rangle^k \|_* \right) \|\mb{w}_0\|_*
	\end{equation}
	for all $i \leq t$. With the repeated application of Lemma~\ref{lem:tensor_norm_upper}, we have
	\[
		\| \langle (\mb{A}^l_r)^* \mb{A}^l, (M^*)^{\otimes l} \rangle^k \|_* \leq (\|U\|_*)^{kl} \leq \left(  r_U^{l} \sigma_1^l(U) \right)^k
	\]
	Therefore, substituting back into \eqref{eq:tildew_nuclear_upper} gives
	\[
		\|\mb{\tilde w}_{i-1} \|_* \leq \left( \sum_{k=0}^t \binom{t}{k} \eta^k \left( r^l r_U^{l} \sigma_1^l(U) \right)^k \right) \|\mb{w}_0\|_* = (1+\tilde \eta \sigma_1^l(U))^{i-1} \|\mb{w}_0\|_*
	\]
	
	Next, plugging the above preparatory results into \eqref{eq:et_norm_upper}, we have that 
	\begin{align*}
	\|\mb{E}_t\|_S &\leq \sum_{i=1}^t (1+\eta \sigma_1^l(U))^{t-i} \eta  (r^2 n L_s)^{l/2} \|\mb{w}_{i-1}\|^3_* \\
	&\leq \sum_{i=1}^t (1+\eta \sigma_1^l(U))^{t-i} \eta  (r^2 n L_s)^{l/2} 8 \|\mb{\tilde w}_{i-1}\|^3_*\\
	&\leq 8 \sum_{i=1}^t (1+\eta \sigma_1^l(U))^{t-i} \eta  (r^2 n L_s)^{l/2} (1+\tilde \eta \sigma_1^l(U))^{3i-3} \|\mb{w}_0\|^3_* \\
	&\leq 8 \epsilon^3 \eta  (r^2 n L_s)^{l/2} \sum_{i=1}^t (1+\tilde \eta \sigma_1^l(U))^{t-i}  (1+\tilde \eta \sigma_1^l(U))^{3i-3} \\
	& = 8 \epsilon^3 \|x_0^{\otimes l}\|_*^3 \eta  (r^2 n L_s)^{l/2} (1+\tilde \eta \sigma_1^l(U))^{t-1} \sum_{i=1}^t   (1+\tilde \eta \sigma_1^l(U))^{2i-2} \\
	& = 8 \epsilon^3 \|x_0^{\otimes l}\|_*^3\eta  (r^2 n L_s)^{l/2} (1+\tilde \eta \sigma_1^l(U))^{t-1} \frac{(1+\tilde \eta \sigma_1^l(U))^{2t}-1}{(1+\tilde \eta \sigma_1^l(U))^{2}-1} \ \text{(geometric sum)} \\
	& \leq 8 \epsilon^3 \|x_0^{\otimes l}\|_*^3 \eta  (r^2 n L_s)^{l/2} (1+\tilde \eta \sigma_1^l(U))^{t-1} (1+\tilde \eta \sigma_1^l(U))^{2t} \\
	& \leq \frac{8 \eta}{\tilde \eta \sigma_1^l(U)} \epsilon^3  (r^2 n L_s)^{l/2} (1+\tilde \eta \sigma_1^l(U))^{3t} \|x_0^{\otimes l}\|_*^3 \\
	& = \frac{r^l 8}{r_U^l \sigma_1^l(U)} \epsilon^3 (n L_s)^{l/2} (1+\tilde \eta \sigma_1^l(U))^{3t} \|x_0^{\otimes l}\|_*^3
\end{align*}
	proving the original claim of this lemma \eqref{eq:E_t_spectral}. Now, we give a lower bound on $t_s$. By recalling the breakdown \eqref{eq:tildew_breakdown}, we have
	\begin{equation}\label{eq:tildew_spectral_lower}
		\begin{aligned}
			\|\mb{\tilde w}_t\|_* &\geq \|\mb{\tilde w}_t\|_S \geq \langle \mb{\tilde w}_t, v_1^{\otimes l} \rangle \\
		&= \epsilon \sum_{k=0}^t \binom{t}{k} \eta^k \left[| v_1^\top \langle \mb{A}^*_r \mb{A}, M^* \rangle^k x_0 | \right]^{ l} \\
		&= \epsilon \sum_{k=0}^t \binom{t}{k} \eta^k \left[ | v_1^\top U^k x_0 | \right]^{l} \\
		&= \epsilon \sum_{k=0}^t \binom{t}{k} \eta^k (| \sigma_1^k(U) v_1^\top  x_0 |)^{l} = \epsilon |v_1^\top  x_0 |^{l} (1+\eta \sigma_1^l(U))^t
		\end{aligned}
	\end{equation}
	with $v_1$ being the first singular vector of $I_r \oslash U$. Since the sensing matrices are assumed to be symmetric, $U$ is also symmetric, hence the singular vectors of $U^k$ coincide with those of $U$. By \eqref{eq:E_t_spectral}, we also know
	\[
		\frac{\|\mb{\tilde w}_{t} - \mb{w}_t\|_*}{ \|\mb{\tilde w}_{t} \|_*} \leq \frac{r^l 8}{r_U^l \sigma_1^l(U)} \epsilon^2 \|x_0^{\otimes l}\|_*^3 \frac{n^{l/2}}{(v_1^\top  x_0 )^{l}} L_s^{l/2} \frac{(1+\tilde \eta \sigma_1^l(U))^{3t}}{(1+\eta \sigma_1^l(U))^t}
	\]
	Therefore, for \eqref{eq:ts_condition} to hold true, we need the RHS of the above equation to be smaller than 1, meaning that
	\[
		3t \ln(1+\tilde \eta \sigma_1^l(U)) \leq \ln\left( \frac{r_U^l \sigma_1^l(U)}{8 r^l \epsilon^2 L_s^{l/2} \|x_0^{\otimes l}\|_*^3} \frac{(v_1^\top  x_0 )^{l}}{n^{l/2}}\right) + t \ln(1+\eta \sigma_1^l(U))
	\]	
	This further implies that for \eqref{eq:ts_condition} to hold, $t$ should satisfy
	\[
		t < \frac{\ln\left( \frac{r_U^l \sigma_1^l(U)}{8 r^l \epsilon^2 L_s^{l/2} \|x_0^{\otimes l}\|_*^3} \frac{(v_1^\top  x_0 )^{l}}{n^{l/2}}\right)}{3\ln(1+\tilde \eta \sigma_1^l(U)) - \ln(1+\eta \sigma_1^l(U))} <  \frac{\ln\left( \frac{r_U^l \sigma_1^l(U)}{8 r^l \epsilon^2 L_s^{l/2} \|x_0^{\otimes l}\|_*^3} \frac{(v_1^\top  x_0 )^{l}}{n^{l/2}}\right)}{2\ln(1+\tilde \eta \sigma_1^l(U))}
	\]
	which after rearrangement gives \eqref{eq:ts_bound}.
\end{proof}

Now, we present the proof of Lemma~\ref{lem:w_eig_ratio}.
\begin{proof}[Proof of Lemma~\ref{lem:w_eig_ratio}]
	Using the tensor Weyl's inequality (Lemma~\ref{lem:tensor_weyl}), we have that 
	\begin{align}
		 \lambda^v_{2}(\mb{w}_t) \leq \lambda^v_{2}(\mb{\tilde w}_t) + \|\mb{E}_t\|_S \\
		\lambda_1^v(\mb{w}_t) \geq \lambda_1^v(\mb{\tilde w}_t) - \|\mb{E}_t\|_S
	\end{align}
	The only remaining part of the proof is the characterization of $\lambda_1^v(\mb{\tilde w}_t)$ and $\lambda^v_{2}(\mb{\tilde w}_t)$. The first term is easy because we already have the characterization from the proof of Lemma~\ref{lem:E_t_spectral}, with \eqref{eq:tildew_spectral_lower} giving rise to
	\[
		\|\mb{\tilde w}_t\|_S \geq \epsilon |v_1^\top  x_0 |^{l} (1+\eta \sigma_1^l(U))^t
	\]
	Also, by the definition of v-eigenvalues and \eqref{eq:tildew_breakdown}, we have that 
	\begin{align*}
		\lambda^v_{2}(\mb{\tilde w}_t) &=  \max_{\substack{V \\ \dim(V)=2}} \min_{\substack{v \in V \\ \|v\|_2=1}} \epsilon \sum_{k=0}^t \binom{t}{k} \eta^k \left[| v^\top \langle \mb{A}^*_r \mb{A}, M^* \rangle^k x_0| \right]^{ l} \\
		&=\epsilon \|x_0\|_2^l \max_{\substack{V \\ \dim(V)=2}} \min_{\substack{v \in V \\ \|v\|_2=1}} \sum_{k=0}^t \binom{t}{k} \eta^k   |v^\top U^k \frac{x_0}{\|x_0\|_2} |^{ l} \\
		&\leq \epsilon \|x_0\|_2^l \max_{\substack{V \\ \dim(V)=2}} \min_{\substack{v \in V \\ \|v\|_2=1}} \sum_{k=0}^t \binom{t}{k} \eta^k   |v^\top U^k v |^{ l} \\
		& = \epsilon \|x_0\|_2^l \sum_{k=0}^t \binom{t}{k} \eta^k   |v_2^\top U^k v_2 |^{ l} \\
		& = \epsilon \|x_0\|_2^l \sum_{k=0}^t \binom{t}{k} \eta^k |\sigma_2^k(U)|^l \\
		& = \epsilon \|x_0\|_2^l (1+\eta \sigma_2^l(U))^t
	\end{align*}
	where $v_2$ is the singular vector associated with $\sigma_2^k(U) \ \forall k \in [t]$. Finally, combining the above equations yields \eqref{eq:w_eig_separation} after rearrangements.
\end{proof}

Next, we present a supporting lemma which explains that Gaussian concentration is suited for our purpose.
\begin{lemma}\label{lem:gauss_init_value}
	Let $x_0 = v_1 +g \in \RR^{nr}$, where $g$ is a vector with each entry being i.i.d sampled from Gaussian distribution $\mathcal{N}(0,\rho)$. For some universal constant $C$, the follwoing inequalities hold:
	\begin{align*}
		&\mathbb{P}\left[ |v_1^\top x_0|^l \geq (1 - \mathcal{O}(\sqrt{\rho}))^l \right] \geq 1-2\exp(-C/\rho), \\
		&\mathbb{P}\left[ \|x_0\|_2^l \leq ( \sqrt{1+\rho^2nr} + \mathcal{O}(\rho^{3/2}))^l \right] \geq 1-2\exp(-C/\rho)
	\end{align*}
\end{lemma}
\begin{proof}[Proof of Lemma~\ref{lem:gauss_init_value}]
	We know that 
	\begin{align*}
		&|v_1^\top x_0| = |1+ v_1^\top g| \geq 1 - |v_1^\top x_0|
	\end{align*}
	Theorem 2.6.3 of \cite{vershynin2018high} (general Hoeffding's) gives that with probability at least $1- 2\exp(-t^2/\rho^2)$,
	\[
		|v^\top g| \leq t \quad \forall \|v\|_2=1
	\]
	which leads to the first concentration bound after substituting $t = \mathcal{O}(\sqrt{\rho})$ with some constant $c_1$. Then, Theorem 3.1.1 in \cite{vershynin2018high} gives
	\[
		\mathbb{P}\left[|\|x_0\|_2 - \sqrt{1+\rho^2 nr}| \leq t \right] \geq 1- 2 \exp(-c_2 t^2/\rho^4)
	\]
	for $g \sim \mathcal{N}(0, \rho I_{nr})$ and some constant $c_2$. This is because $\mathbb{E}[\|x_0\|^2_2 ]= 1+\rho^2 nr$. Substituting  $t = \mathcal{O}(\rho^{3/2})$ yields that
	\[
		\mathbb{P}\left[\|x_0\|_2 \leq \sqrt{1+\rho^2nr} + \mathcal{O}(\rho^{3/2}) \right]  \geq  1- 2\exp(-c_2/\rho)
	\]
	which results in the second bound. Now, we choose $C = \min\{c_1,c_2\}$.
\end{proof}

Then, we prove our main theorem of this section.
\begin{proof}[Proof of Theorem~\ref{thm:implicit_bias_gd}]
		First, set $2\zeta = \kappa$, implying that $\zeta < 1/2$. We aim to derive sufficient conditions for the following inequalities to hold:
		\begin{align}
				\lambda_2^v(\mb{\tilde w}_t) \leq \frac{\zeta}{2} \lambda_1^v(\mb{\tilde w}_t), \label{eq:thm_implicit_help1}\\
				\|\mb{E}_t\|_s \leq \frac{\zeta}{2} \lambda_1^v(\mb{\tilde w}_t) \label{eq:thm_implicit_help2}
		\end{align}
		By recalling Lemma~\ref{lem:w_eig_ratio}, a sufficient condition for \eqref{eq:thm_implicit_help1} is that 
		\[
			\epsilon \|x_0\|_2^l (1+\eta \sigma_2^l(U))^t \leq \frac{\zeta}{2} \epsilon |v_1^\top  x_0 |^{l} (1+\eta \sigma_1^l(U))^t
		\]
		implying that
		\[
			\frac{2 \|x_0\|_2^l}{\zeta |v_1^\top  x_0 |^{l}} \leq \left( \frac{ 1+\eta \sigma_1^l(U)}{1+\eta \sigma_2^l(U)}\right)^t
		\]
		which after rearrangements gives $t \geq t(\zeta,l)$, as defined in \eqref{eq:t_kappa_l}. Then, we obtain a sufficient condition for \eqref{eq:thm_implicit_help2}, which by Lemma~\ref{lem:E_t_spectral} is
		\begin{equation}\label{eq:thm_implicit_help3}
			\frac{8 r^l}{r_U^l \sigma_1(U)^l} \epsilon^3 (n L_s)^{l/2} (1+\tilde \eta \sigma_1(U)^l)^{3t} \|x_0^{\otimes l}\|_*^3 \leq \frac{2}{\zeta} \epsilon |v_1^\top  x_0 |^{l} (1+\eta \sigma_1^l(U))^t 
		\end{equation}
		contingent on the fact that $t \leq t_s$. Therefore, before going further, we need to verify that $t(\zeta,l) \leq t_s$ for some small enough $\epsilon$. \eqref{eq:ts_bound} implies that a sufficient condition is
		\begin{align*}
			\ln\left( \frac{2\|x_0\|^l_2}{\zeta |v_1^\top x_0|^l}\right) \ln\left( \frac{ 1+\eta \sigma_1^l(U)}{1+\eta \sigma_2^l(U)} \right)^{-1} \leq \frac{\ln\left( \frac{\sigma_1^l(U)r_U^l}{8 r^l L_s^{l/2} \|x_0^{\otimes l}\|_*^3 \epsilon^2} \frac{|x_0^\top v_1|^l}{ n^{l/2}}\right)}{2\ln(1+\tilde \eta \sigma_1^l(U))}
		\end{align*}
		Additionally, by leveraging the identity $x/(1+x) \leq \ln(1+x) \leq x$, we derive the following identity 
		\begin{equation} \label{eq:xi_def}
			\frac{\ln(1+\tilde \eta \sigma_1^l(U))}{\ln\left( \frac{ 1+\eta \sigma_1^l(U)}{1+\eta \sigma_2^l(U)} \right)^{-1}} \leq \frac{r_U^l(1+\eta \sigma_1^l(U))}{1-(\sigma_2(U)/\sigma_1(U))^l} \coloneqq \Xi
		\end{equation}
		Hence,
		\[
			2 \ln\left( \frac{2\|\mb{w}_0\|^l_2}{\zeta |v_1^\top x_0|^l}\right) \Xi \leq \ln\left( \frac{\sigma_1^l(U)r_U^l}{8 r^l L_s^{l/2} \|x_0^{\otimes l}\|_*^3 \epsilon^2} \frac{|x_0^\top v_1|^l}{ n^{l/2}}\right)
		\]
		and after rearrangement gives 
		\begin{equation}\label{eq:thm_implicit_help4}
			\epsilon^2 \leq \frac{\sigma_1^l(U)r_U^l}{8 (r^2nL_s)^{l/2}} \frac{|x_0^\top v_1|^l}{\|x_0^{\otimes l}\|_*^3} \left( \frac{2\|x_0\|^l_2}{\zeta |v_1^\top x_0|^l}\right)^{-\Xi}
		\end{equation}
		Notice that all of the above terms are independent of $\epsilon$, and are positive. Therefore, a small enough $\epsilon$ exists. Also notice that a smaller step-size $\eta$ will yield a loser bound on $\epsilon$ through the dependence of $\Xi$.
		Now, consider \eqref{eq:thm_implicit_help3} again. Since $T$ is finite, a sufficient condition for \eqref{eq:thm_implicit_help3} is
		\begin{equation}\label{eq:thm_implicit_help5}
			\epsilon^2 \leq \zeta \frac{r_U^l \sigma_1(U)^l}{16 (r^2n L_s)^{l/2} } \frac{|v_1^\top  x_0 |^{l}}{\|x_0^{\otimes l}\|_*^3} \left(\frac{1+\eta \sigma_1^l(U)}{(1+\tilde \eta \sigma_1(U)^l)^3}\right)^T
		\end{equation}
		which can again be achieved by setting a small enough $\epsilon$, since all other terms are positive and not dependent on it. In summary, if we choose a small constant $\epsilon$ satisfying both \eqref{eq:thm_implicit_help4} and \eqref{eq:thm_implicit_help5}, and if $t_s \geq t_T$ (which again can be achieved via a sufficiently small $\epsilon$), it is already sufficient for both \eqref{eq:thm_implicit_help1} and \eqref{eq:thm_implicit_help2} to hold, thereby giving:
		\[
			\frac{\lambda_2^v(\mb{\tilde w}_t)+\|\mb{E}_t\|_S}{\lambda_1^v(\mb{\tilde w}_t)} \leq \zeta
		\]
		If $\zeta < 1/2$, this further implies
		\[
			\lambda_1^v(\mb{\tilde w}_t) > 2\lambda_2^v(\mb{\tilde w}_t)+2\|\mb{E}_t\|_S \implies \|\mb{E}_t\|_S \leq \frac{1}{2} \lambda_1^v(\mb{\tilde w}_t) - \lambda_2^v(\mb{\tilde w}_t) \leq \frac{1}{2} \lambda_1^v(\mb{\tilde w}_t)
		\]
		As a result,
		\[
			\frac{\lambda^v_{2}(\mb{w}_t)}{\lambda^v_1(\mb{w}_t)} \leq \frac{\lambda_2^v(\mb{\tilde w}_t)+\|\mb{E}_t\|_S}{\lambda_1^v(\mb{\tilde w}_t)-\|\mb{E}_t\|_S} \leq \frac{\zeta \lambda_1^v(\mb{\tilde w}_t)}{\lambda_1^v(\mb{\tilde w}_t)/2} = 2\zeta
		\]
		which proves \eqref{eq:w_eig_kappa}. 
	\end{proof}
	
	Theorem~\ref{thm:implicit_bias_gd} can also be improved via  Lemma~\ref{lem:gauss_init_value} as stated below. 
	\begin{corollary}[Corollary to Theorem~\ref{thm:implicit_bias_gd}]\label{cor:asymp_implicit_bias}
		Consider the optimization problem and the GD trajectory given in Theorem~\ref{thm:implicit_bias_gd}. If additionally $x_0 = v_1 +g \in \RR^{nr}$ and $g \sim \mathcal{N}(0, \rho I_{nr})$, then
		\begin{equation} \label{eq:magnitude_t_w_eig}
			\frac{\lambda^v_{2}(\mb{w}_t)}{\lambda^v_1(\mb{w}_t)} \leq \kappa \quad \text{for} \ \ t \asymp \ln\left( \frac{1}{\kappa}\right) \ln\left( \frac{ 1+\eta \sigma_1^l(U)}{1+\eta \sigma_2^l(U)} \right)^{-1}
		\end{equation}
		provided that
		\begin{equation}
			\epsilon \asymp \sqrt{\kappa/2} \frac{(\sigma_1(U)r_U)^{l/2}}{4 (r^2nL_s)^{l/4}} (\frac{\kappa}{4})^{3\Xi/2}, \quad \text{where} \ \Xi \coloneqq \frac{r_U^l(1+\eta \sigma_1^l(U))}{1-(\sigma_2(U)/\sigma_1(U))^l} 
		\end{equation}
		with probability at least $1-2\exp(-C/\rho)$ for some universal constant $C$ as $\rho \rightarrow 0$, where $\sigma_1(U)$ and $\sigma_2(U)$ are the first two singular values of $U = \langle \mathbf{A}_r, b \rangle_3$, with $v_1$ being the associated singular vector of $\sigma_1(U)$ ($\asymp$ denotes "asymptotic to", meaning that the two terms of both sides of this symbol are of the same order of magnitude).
	\end{corollary}
	\begin{proof}[Proof of Corollary~\ref{cor:asymp_implicit_bias}]
		The proof is similar to that of Theorem~\ref{thm:implicit_bias_gd} (note $\zeta = \kappa/2$), and therefore we only highlight the difference. We know that \eqref{eq:thm_fop_highr_help1} holds true if
		\begin{align*}
			t \geq \ln\left( \frac{2\|x_0\|^l_2}{\zeta |v_1^\top x_0|^l}\right) \ln\left( \frac{ 1+\eta \sigma_1^l(U)}{1+\eta \sigma_2^l(U)} \right)^{-1},  \\
			\epsilon^2 \leq \frac{\sigma_1^l(U)r_U^l}{8 (r^2nL_s)^{l/2}} \frac{|x_0^\top v_1|^l}{\|x_0^{\otimes l}\|_*^3} \left( \frac{2\|x_0\|^l_2}{\zeta |v_1^\top x_0|^l}\right)^{-\Xi}
		\end{align*}
		It results from Lemma~\ref{lem:gauss_init_value} that for our choice of initialization, we have that
		\[
			\|x_0\|^l_2 \asymp \|v_1^\top x_0|^l \asymp 1
		\]
		with probability at least $1-2\exp(-C /\rho)$. Thus, as long as 
		\begin{align}
			&t \asymp \ln\left( \frac{2}{\zeta}\right) \ln\left( \frac{ 1+\eta \sigma_1^l(U)}{1+\eta \sigma_2^l(U)} \right)^{-1} \coloneqq t_*, \\
			&\epsilon \asymp \frac{(\sigma_1(U)r_U)^{l/2}}{2\sqrt{2} (r^2nL_s)^{l/4}} \left( \frac{2}{\zeta}\right)^{-\Xi/2} \label{eq:cor_implicit_help1}
		\end{align}
		\eqref{eq:thm_fop_highr_help1} will hold with high probability. Next, in order for \eqref{eq:thm_implicit_help2} to hold for $t \asymp t_*$, we know that 
		\[
			\epsilon^2 \leq \zeta \frac{r_U^l \sigma_1(U)^l}{16 (r^2n L_s)^{l/2} } \frac{|v_1^\top  x_0 |^{l}}{\|x_0^{\otimes l}\|_*^3} \left(\frac{1+\eta \sigma_1^l(U)}{(1+\tilde \eta \sigma_1(U)^l)^3}\right)^{t_*}
		\]
		Via the same order of magnitude argument, we know that the following condition is sufficient for \eqref{eq:thm_implicit_help2}:
		\[
			\epsilon \asymp \sqrt{\zeta} \frac{(\sigma_1(U)r_U)^{l/2}}{4 (r^2nL_s)^{l/4}} \left(\frac{1+\eta \sigma_1^l(U)}{(1+\tilde \eta \sigma_1(U)^l)^3}\right)^{t_*/2}
		\]
		Now,
		\begin{align*}
			\left(\frac{1+\eta \sigma_1^l(U)}{(1+\tilde \eta \sigma_1(U)^l)^3}\right)^{t_*/2} &\geq \left[ \frac{1}{(1+\tilde \eta \sigma_1(U)^l)^3 } \right]^{t_*/2}\\
			&=\exp\left( -\frac{3t_*}{2} \ln(1+\tilde \eta \sigma_1(U)^l) \right) \\
			&=\exp\left( -\frac{3}{2} \ln(\frac{2}{\zeta}) \frac{\ln(1+\tilde \eta \sigma_1(U)^l)}{\ln\left( \frac{ 1+\eta \sigma_1^l(U)}{1+\eta \sigma_2^l(U)} \right)} \right) \\
			& \geq \exp\left( -\frac{3}{2} \ln(\frac{2}{\zeta}) \Xi \right) = (\frac{\zeta}{2})^{3\Xi/2}
		\end{align*}
		where the second equality follows from the substitution of $t_*$, and the last inequality follows from \eqref{eq:xi_def}. As a result,
		\begin{equation}\label{eq:cor_implicit_help2}
			\epsilon \asymp \sqrt{\zeta} \frac{(\sigma_1(U)r_U)^{l/2}}{4 (r^2nL_s)^{l/4}} (\frac{\zeta}{2})^{3\Xi/2}
		\end{equation}
		Therefore, taking the minimum of \eqref{eq:cor_implicit_help1} and \eqref{eq:cor_implicit_help2}, we know that 
		\begin{equation}
			\epsilon \asymp \sqrt{\zeta} \frac{(\sigma_1(U)r_U)^{l/2}}{ (r^2nL_s)^{l/4}} (\frac{\zeta}{2})^{3\Xi/2}
		\end{equation}
		is sufficient for \eqref{eq:thm_implicit_help1} and \eqref{eq:thm_implicit_help2}, leading to \eqref{eq:magnitude_t_w_eig} via the same steps in the proof of Theorem~\ref{thm:implicit_bias_gd}.
	\end{proof}
	
\newpage
\section{Additional Details for Properties of Approximate Rank-1 Tensors}\label{sec:app_rank1_properties}

We start with the proof of Proposition~\ref{prop:tensor_rank1_breakdown}.
\begin{proof}[Proof of Proposition~\ref{prop:tensor_rank1_breakdown}]
	Given a symmetric tensor $\mb{w}$, it can be decomposed as
	\[
		\mb{w} = \sum_{i=1}^{r_w} \lambda_i x_i^{\otimes l}
	\]
	where $r_w$ is $\mb{w}$'s symmetric rank. Now, consider the vector $w_s \in \RR^n$ that attains the spectral norm, meaning that $\langle \mb{w}, w_s^{\otimes l} \rangle = \lambda_1^v(\mb{w})$. One can decompose each $x_i^{\otimes l}$ into a parallel component and an orthogonal component. To be more specific,
	\[
		x_i = x_i^s + x_i^\perp \implies x_i^{\otimes l} = (x_i^s)^{\otimes l} + \sum_{j=1}^{2^l-1} \underbrace{x_i^\perp \otimes \dots \otimes x_i^\perp}_{j} \otimes \underbrace{x_i^s \otimes \dots \otimes x_i^s}_{l-j}
	\]
	and it is apparent that the second term is orthogonal to $w_s^{\otimes l}$ via Lemma~\ref{lem:kron_iden}. Therefore, we just organize all components $w_s^{\otimes l}$ together and all orthogonal components together. By definition, the parallel component has the magnitude $\lambda_1^v(\mb{w})$. Also, by the definition of v-eigenvalues, $\|\mb{w}^\dagger\|_S \leq \lambda^v_2(\mb{w}_t)$ since otherwise the dominant direction of $\mb{w}^\dagger$ will just become the second eigenvector of $\mb{w}$.
\end{proof}

We now provide the proof of Proposition~\ref{prop:kappa_fop}.
\begin{proof}[Proof of Proposition~\ref{prop:kappa_fop}]
	According to \eqref{eq:focp_lifted_highr}, the gradient of \eqref{eq:lifted_main_highr} with respect to $\mb{w}$ can be expressed as
	\begin{equation}\label{eq:prop_kappa_fop_help1}
		\nabla h^l(\mb{w}) = \langle \langle (\mb{A}^l_r)^* \mb{A}^l, \langle \mb{P}(\mb{w}), \mb{P}(\mb{w}) \rangle_{2*[l]} - (M^*)^{\otimes l} \rangle, \mb{w} \rangle_{2*[l]}
	\end{equation}
	where $(\mb{A}^l_r)^* \mb{A}^l$ is defined in \eqref{eq:A_r_def}. In light of \eqref{eq:tensor_rank1_breakdown}, one can write
	\begin{align*}
		&\langle \mb{P}(\mb{w}), \mb{P}(\mb{w}) \rangle_{2*[l]} = \langle  \langle \mb{P}^{\otimes l}, \mb{P}^{\otimes l} \rangle_{2*[l]}, \mb{w} \otimes \mb{w}  \rangle_{3,4,7,8,\dots,4l-1,4l} \\
		= & \underbrace{ \langle  \langle \mb{P}^{\otimes l}, \mb{P}^{\otimes l} \rangle, \mb{w}_\sigma \otimes \mb{w}_\sigma  \rangle}_{\mb{a}_1} + 2 \underbrace{\langle  \langle \mb{P}^{\otimes l}, \mb{P}^{\otimes l} \rangle, \mb{w}_\sigma \otimes \mb{w}^\dagger  \rangle}_{\mb{a}_2} + \underbrace{\langle  \langle \mb{P}^{\otimes l}, \mb{P}^{\otimes l} \rangle, \mb{w}^\dagger \otimes \mb{w}^\dagger  \rangle}_{\mb{a}_3}
	\end{align*}
	where $\mb{w}_\sigma = \lambda_1^v(\mb{w}) \hat w^{\otimes l}$. Note that we have dropped the subscripts from the second line and henceforth for sake of simplicity. By using this logic, \eqref{eq:prop_kappa_fop_help1} can be written as
	\begin{align*}
		\nabla h^l(\mb{w}) = \underbrace{\langle \langle (\mb{A}^l_r)^* \mb{A}^l, \langle \mb{a}_1 - (M^*)^{\otimes l} \rangle, \mb{w}_{\sigma} \rangle}_{\mb{h}_1} + \mb{h}_2
	\end{align*}
	where
	\begin{align*}
		\mb{h}_2 = &\langle \langle (\mb{A}^l_r)^* \mb{A}^l,  \mb{a}_1  \rangle, \mb{w}^\dagger \rangle+ \langle \langle (\mb{A}^l_r)^* \mb{A}^l,  \mb{a}_2  \rangle, \mb{w}_{\sigma} \rangle + \langle \langle (\mb{A}^l_r)^* \mb{A}^l, \mb{a}_2  \rangle, \mb{w}^\dagger \rangle + \\
		&\langle \langle (\mb{A}^l_r)^* \mb{A}^l, \mb{a}_3  \rangle, \mb{w}_{\sigma} \rangle + \langle \langle (\mb{A}^l_r)^* \mb{A}^l,  \mb{a}_3  \rangle, \mb{w}^\dagger \rangle - \langle \langle (\mb{A}^l_r)^* \mb{A}^l,  (M^*)^{\otimes l}  \rangle, \mb{w}^\dagger \rangle
	\end{align*}
	The first term can be analyzed as
	\[
		\langle \langle (\mb{A}^l_r)^* \mb{A}^l,  \mb{a}_1  \rangle, \mb{w}^\dagger \rangle = \langle (\mb{A}^l_r)^* \mb{A}^l, \langle \langle \mb{P}^{\otimes l}, \mb{P}^{\otimes l} \rangle, \mb{w}_\sigma \otimes \mb{w}_\sigma \otimes \mb{w}^\dagger \rangle \rangle
	\]
	and by Lemma~\ref{lem:tensor_norm_upper}, we have that 
	\begin{equation}\label{eq:prop_kappa_fop_help2}
		\begin{aligned}
			\|\langle \langle (\mb{A}^l_r)^* \mb{A}^l,  \mb{a}_1  \rangle, \mb{w}^\dagger \rangle\|_S &\leq \|\langle \langle \mb{P}^{\otimes l}, \mb{P}^{\otimes l} \rangle, \mb{w}_\sigma \otimes \mb{w}_\sigma \otimes \mb{w}^\dagger \rangle\|_S \|(\mb{A}^l_r)^* \mb{A}^l\|_* \\
			& = \| \langle \mb{P}(\mb{w}_\sigma), \mb{P}(\mb{w}_\sigma) \rangle \otimes \mb{w}^\dagger \|_S \|(\mb{A}^l_r)^* \mb{A}^l\|_* \\
		&\leq \lambda_1^v(\mb{w})^2 \|\mb{w}^\dagger\|_S  \|(\mb{A}^l_r)^* \mb{A}^l\|_* \leq \kappa \lambda_1^v(\mb{w})^3  r^l \| \mb{A}^*\mb{A}\|_*^l
		\end{aligned}
	\end{equation}
	The second inequality follows form that for all $u_1 \in \RR^n$ and $u_2 \in \RR^{nr}$ such that $\|u_1\|_2=1$ and $ \|u_2\|_2=1$:
	\begin{align*}
		\| \langle \mb{P}(\mb{w}_\sigma), \mb{P}(\mb{w}_\sigma) \rangle \otimes \mb{w}^\dagger \|_S &= \max_{u_1, u_2} \langle \langle \mb{P}(\mb{w}_\sigma), \mb{P}(\mb{w}_\sigma) \rangle \otimes \mb{w}^\dagger, u_1^{\otimes 2l} \otimes u_2^{\otimes l}  \rangle \\
		&\leq \lambda_1^v(\mb{w})^2 (u^\top \mat(\hat x) \mat(\hat x)^\top u)^l \|\mb{w}^\dagger\|_S \\
		&\leq \lambda_1^v(\mb{w})^2 \sigma_{\max}(\mat(\hat x))^{2l} \|\mb{w}^\dagger\|_S \\
		&\leq \lambda_1^v(\mb{w})^2 \|\hat x\|^{2l}_2 \|\mb{w}^\dagger\|_S \\
		& = \lambda_1^v(\mb{w})^2 \|\mb{w}^\dagger\|_S 
	\end{align*}
	Repeating this process leads to
	\begin{equation}\label{eq:prop_kappa_fop_help3}
		\|\mb{h}_2\|_S \leq (3\kappa + 3 \kappa^2 +\kappa^3 + \kappa \|M^*\|^2_F )\lambda_1^v(\mb{w})^3 r^l \| \mb{A}^*\mb{A}\|_*^l
	\end{equation}
	Similarly, $\|\mb{h}_1\|_S  = \mathcal{O}(\lambda_1^v(\mb{w})^3 r^l \| \mb{A}^*\mb{A}\|_*^l)$. Now, if we assume that $\mb{w}$ is an FOP of \eqref{eq:lifted_main_highr}, it means that $\nabla h^l(\mb{w}) = 0$, further implying $\|\nabla h^l(\mb{w})\|_S = 0$, and by reverse triangle inequality,
	\[
		0 = \|\nabla h^l(\mb{w})\|_S \geq |\|\mb{h}_1\|_S - \|\mb{h}_2\|_S|
	\]
	which means that $\|\mb{h}_1\|_S = \|\mb{h}_2\|_S$. Since there always exits a small enough $\kappa$ such that $\|\mb{h}_2\|_S = c \|\mb{h}_1\|_S$ with $c < 1$, and therefore the only possibility that the above inequality holds true is that $\|\mb{h}_1\|_S = \|\mb{h}_2\|_S=0$. This implies
	\[
		\langle \mb{h}_1, u^{\otimes l} \rangle = (\langle \mb{A}, \mat(w_s) \mat(w_s)^\top - M^* \rangle^\top \langle \mb{A}, \mat(w_s) \mat(u)^\top  \rangle)^l = 0 \quad \forall u \in \RR^{nr}
	\]
	which is equivalent to the FOP condition for \eqref{eq:unlifted_main}, which is \eqref{eq:focp_unlifted}, meaning that $\mat(w_s) \in \RR^{n \times r}$ is an FOP of \eqref{eq:unlifted_main}. Note that we can always scale $\mb{A}$ and $b$ together so that $\| \mb{A}^*\mb{A}\|_*^l$ can be normalized to 1.
\end{proof}

Finally, we prove the main result of this paper.
\begin{proof}[Proof of Theorem~\ref{thm:approx_rank1_saddle}]
	We consider the SOP condition for \eqref{eq:lifted_main_highr}, which is \eqref{eq:socp_lifted_highr} for some rank-1 tensor $\Delta$.  We can express it as
	\begin{align*}
		\nabla^2 h^l(\mb{\hat w})[\Delta,\Delta] = &2 \underbrace{\langle \nabla f^l(\langle \mb{P}(\mb{\hat w}), \mb{P}(\mb{\hat w}) \rangle_{2*[l]}), \langle \mb{P}(\Delta), \mb{P}(\Delta) \rangle_{2*[l]}}_{\mb{a}_1(\mb{\hat w})} + \\
			&\underbrace{\| \langle \mb{A}^{\otimes l} , \langle \mb{P}(\mb{\hat w}), \mb{P}(\Delta) \rangle_{2*[l]} + \langle \mb{P}(\Delta), \mb{P}(\mb{\hat w}) \rangle_{2*[l]} \rangle \|_F^2}_{\mb{a}_2(\mb{\hat w})}
	\end{align*}
	Let $\Delta$ be defined identically to that in the proof of Theorem~\ref{thm:socp_highr}, meaning that $\Delta = \vecc(U)^{\otimes} \coloneqq u^{\otimes l}$. By the same logic of \eqref{eq:prop_kappa_fop_help1}, we have that 
	\begin{align*}
		\mb{a}_1(\mb{\hat w}) &= \langle \langle (\mb{A}^l_r)^* \mb{A}^l, \langle \mb{P}(\mb{\hat w}), \mb{P}(\mb{\hat w}) \rangle - (M^*)^{\otimes l} \rangle, \Delta \otimes \Delta \rangle \\
		&= \underbrace{\langle (\mb{A}^l_r)^* \mb{A}^l, \langle \langle \mb{P}^{\otimes l}, \mb{P}^{\otimes l} \rangle, \mb{\hat w} \otimes \mb{\hat w} \otimes \Delta \otimes \Delta \rangle \rangle}_{\mb{b}_1} - \langle (\mb{A}^l_r)^* \mb{A}^l,  (M^*)^{\otimes l} \otimes \Delta \otimes \Delta \rangle \rangle
	\end{align*}
	Since $\mb{\hat w}$ is a $\kappa$-rank-1 tensor, by denoting $\lambda_S \hat x^{\otimes l} \coloneqq \mb{w}_\sigma$, we represent
	\begin{align*}
		\mb{b}_1 = &\langle (\mb{A}^l_r)^* \mb{A}^l, \langle \langle \mb{P}^{\otimes l}, \mb{P}^{\otimes l} \rangle, \mb{w}_\sigma \otimes \mb{w}_\sigma \otimes \Delta \otimes \Delta \rangle \rangle + \\
		2 &\underbrace{\langle (\mb{A}^l_r)^* \mb{A}^l, \langle \langle \mb{P}^{\otimes l}, \mb{P}^{\otimes l} \rangle, \mb{w}_\sigma \otimes \mb{\hat w}^\dagger \otimes \Delta \otimes \Delta \rangle \rangle}_{\mb{c}_1} +\\
		&\underbrace{\langle (\mb{A}^l_r)^* \mb{A}^l, \langle \langle \mb{P}^{\otimes l}, \mb{P}^{\otimes l} \rangle, \mb{\hat w}^\dagger \otimes \mb{\hat w}^\dagger \otimes \Delta \otimes \Delta \rangle \rangle}_{\mb{c}_2}
	\end{align*}
	Hence,
	\[
		\mb{a}_1(\mb{\hat w}) = \mb{a}_1(\mb{w}_\sigma) + 2\mb{c}_1 + \mb{c}_2
	\]
	Now, we turn to $\mb{a}_2(\mb{\hat w})$. Since the sensing matrices are assumed to be symmetric, by \eqref{eq:thm_socp_highr_help1}, we have
	\begin{align*}
		\mb{a}_2(\mb{\hat w}) &= 4 \langle \langle (\mb{A}^l_r)^* \mb{A}^l, \langle \mb{P}(\mb{\hat w}), \mb{P}(\Delta) \rangle ,  \Delta \otimes \mb{\hat w} \rangle \\
		&= 4 \underbrace{\langle (\mb{A}^l_r)^* \mb{A}^l, \langle \langle \mb{P}^{\otimes l}, \mb{P}^{\otimes l} \rangle, \mb{\hat w} \otimes \Delta \otimes \mb{\hat w} \otimes \Delta \rangle \rangle}_{\mb{b}_2}
	\end{align*}
	again following the procedures in \eqref{eq:prop_kappa_fop_help1}. Given the decomposition of $\mb{\hat w}$, we decompose $\mb{b}_2$ similarly to $\mb{b_1}$:
	\begin{align*}
		\mb{b}_2 = &\langle (\mb{A}^l_r)^* \mb{A}^l, \langle \langle \mb{P}^{\otimes l}, \mb{P}^{\otimes l} \rangle, \mb{w}_\sigma \otimes \Delta \otimes \mb{w}_\sigma \otimes \Delta \rangle \rangle + \\
		&\underbrace{\langle (\mb{A}^l_r)^* \mb{A}^l, \langle \langle \mb{P}^{\otimes l}, \mb{P}^{\otimes l} \rangle, \mb{w}_\sigma \otimes \Delta  \otimes \mb{\hat w}^\dagger \otimes \Delta + \mb{\hat w}^\dagger \otimes \Delta  \otimes \mb{w}_\sigma \otimes \Delta \rangle \rangle}_{\mb{c}_3} +\\
		&\underbrace{\langle (\mb{A}^l_r)^* \mb{A}^l, \langle \langle \mb{P}^{\otimes l}, \mb{P}^{\otimes l} \rangle, \mb{\hat w}^\dagger  \otimes \Delta \otimes \mb{\hat w}^\dagger \otimes \Delta \rangle \rangle}_{\mb{c}_4}
	\end{align*}
	Combining everything together, we have
	\begin{align*}
		\nabla^2 h^l(\mb{\hat w})[\Delta,\Delta] &= \mb{a}_1(\mb{w}_\sigma) + 2\mb{c}_1 + \mb{c}_2 + \mb{a}_2(\mb{w}_\sigma) + 4 \mb{c}_3 + 4 \mb{c}_4 \\
		&= \nabla^2 h^l(\mb{w}_\sigma)[\Delta,\Delta] + 2\mb{c}_1 + \mb{c}_2 + 4 \mb{c}_3 + 4 \mb{c}_4
	\end{align*}
	In addition, following the same procedures in \eqref{eq:prop_kappa_fop_help2},
	\[
		2\mb{c}_1 + \mb{c}_2 + 4 \mb{c}_3 + 4 \mb{c}_4 \leq ( 10 \kappa + 5\kappa^2) \lambda_S^2 r^l \| \mb{A}^*\mb{A}\|_*^l
	\]
	Now, since $\mb{w}_\sigma$ is a lifted version of FOP for \eqref{eq:unlifted_main} (via Proposition~\ref{prop:tensor_rank1_breakdown}), 
	\[
		\nabla^2 h^l(\mb{w}_\sigma)[\Delta,\Delta] \leq -2G^l + \frac{2}{2^{l-1}} L^l_s \lambda_r(\hat X \hat X^\top)^l
	\]
	where $\hat X = \mat(\hat x)$ and $G \coloneqq - \lambda_{\text{min}}(\nabla f(\hat X \hat X^\top)) \geq 0$. Remember that the choice of $\Delta$ is identical. Therefore, a sufficient condition for $\nabla^2 h^l(\mb{\hat w})[\Delta,\Delta] \leq 0$ is that 
	\[
		2G^l \geq \frac{2}{2^{l-1}} L^l_s \lambda_r(\hat X \hat X^\top)^l + ( 10 \kappa + 5\kappa^2) \lambda_S^2 r^l \| \mb{A}^*\mb{A}\|_*^l
	\] 
	We can derive another sufficient condition to the above inequality, which is 
	\[
		G \geq 2^{1/l-1} L_s \lambda_r(\hat X \hat X^\top) + (5 \kappa + 5\kappa^2/2)^{1/l} \lambda_S^{2/l} r \| \mb{A}^*\mb{A}\|_*
	\]
	since $(a+b)^{1/l} \leq a^{1/l} + b^{1/l}$ for $a,b \geq 0$. Following the steps of the proof of Theorem~\ref{thm:socp_highr}, we obtain that
	\[
		\|M^* - \hat X \hat X^\top \|^2_F > 2^{1/l} \frac{L_s}{\alpha_s} \lambda_r(\hat X \hat X^\top) \tr(M^*) + \mathcal{O}(r \kappa^{1/l})
	\]
	is sufficient. Note that $\| \mb{A}^*\mb{A}\|_*$ can be rescaled to 1 easily. Following the same steps, we can set
	\[
		\beta = \frac{L_s \tr(M^*) \lambda_r(\hat X \hat X^\top)}{\alpha_s \|M^* - \hat X \hat X^\top \|^2_F - \mathcal{O}(r \kappa^{1/l})} 
	\]
	and this leads to the desirable result.
\end{proof}
\newpage

\section{Additional Experiments}\label{sec:app_exp}
In this section, we provide some additional experiments to showcase the algorithmic regularization of GD algorithm in tensor problems like \eqref{eq:lifted_main_highr}.

This section involves the decomposition of tensors along the optimization trajectory using a known algorithm, S-HOPM, as outlined in \cite{kofidis2002best}. The S-HOPM algorithms extract the dominant rank-1 component of a given tensor, so as a first step, we apply this to tensors on the trajectory, and obtain \( \mb{w}_1 \). Subsequently, this component was subtracted from the original tensor, and the extraction procedure was repeated on the resultant tensor \( \mb{w} - \mb{w}_1 \) to obtain a new component \( \mb{w}_2 \). This allows us to directly compute \( \frac{\|\mb{w}_1\|_F}{\|\mb{w}_2\|_F} \), in the hope to approximate $\lambda^v_{2}(\mb{w}_t)/\lambda^v_1(\mb{w}_t)$ for some given $t$ in the trajectory. Note that this procedure mirrors the definition of the variational eigenvalue of tensors defined in Definition~\ref{def:v_eigenvalues}. The main source of inaccuracy is that the S-HOPM algorithm may not find the real dominant rank-1 component, as specified in the original paper. Therefore, the metric we show below only serves as an approximation of $\lambda^v_{2}(\mb{w}_t)/\lambda^v_1(\mb{w}_t)$.

For a practical illustration, we focused on a problem defined in Section 6.1, characterized by a parameter \( n = 8 \). We were particularly interested in observing the evolution of the aforementioned ratio along the optimization trajectory during the process of gradient descent optimization. The results of this analysis are tabulated below:

\[
\begin{array}{|c|c|c|c|c|c|c|c|c|c|}
\hline
\text{iteration} & 20 & 40 & 60 & 80 & 100 & 120 & 140 & 160 & 180 \\
\hline
\epsilon=10^{-5} & 1.16 & 0.95 & 0.82 & 0.05 & 0.03 & 0.018 & 0.026 & 0.028 & 0.013 \\
\hline
\epsilon=10^{-3} & 0.13 & 0.43 & 0.44 & 0.031 & 0.036 & 0.0008 & 0.034 & 0.028 & 0.022 \\
\hline
\epsilon=0.1 & 0.14 & 0.02 & 0.05 & 0.034 & 0.031 & 0.026 & 0.022 & 0.034 & 0.037 \\
\hline
\end{array}
\]

This table exhibits a notable trend where the tensor gradually exhibits more of a "rank-1" nature, aligning with the assertions made in Theorem 1. Interestingly, this behavior is observed across varying initialization scales (\( \epsilon \)), indicating that the phenomenon is not restricted to smaller scales, thus broadening the potential applicability of our findings.

This ratio provides meaningful insights into the training dynamics, which further substantiates the claims made under Theorem~\ref{thm:implicit_bias_gd}.

\newpage

\section{Custom Algorithms}\label{sec:app_algos}

\begin{algorithm}[ht]

\DontPrintSemicolon
\caption{CustomGD Algorithm}

\textbf{Input:} learning\_rate, n, r, l, prob\_params, loss, g\_thres, buffer, beta, gamma, eta\_0\\
\textbf{Initialize variables:} A, b, escape\_saddle, buffer\_limit, buffer\_step

\SetKwFunction{init}{init}
\SetKwFunction{update}{update}

\SetKwProg{Fn}{Function}{}{}

\Fn{\init{starting\_point, lr}}{
    \If{lr $\neq$ 0}{
        learning\_rate $\leftarrow$ lr \tcp{Update learning rate if specified}
    }
    \Return $\{'curr\_iter': 0, 't\_noise': 0, 'curr\_w': starting\_point\}$
}

\Fn{\update{gradients, opt\_state}}{
    curr\_iter $\leftarrow$ opt\_state['curr\_iter'] + 1\;
    t\_noise $\leftarrow$ opt\_state['t\_noise']\;
    curr\_w $\leftarrow$ opt\_state['curr\_w']\;
    \If{$\lVert \text{gradients} \rVert < g\_thres$ and curr\_iter $>$ 100}{
        \If{escape\_saddle}{
            t\_noise $\leftarrow$ curr\_iter\;
            w\_s $\leftarrow$ find rank 1 component of curr\_w using tensor PCA \;
            direction $\leftarrow$ find the escape direction of w\_s \tcp{According to Theorem~\ref{thm:approx_rank1_saddle}}
            this\_eta $\leftarrow$ eta\_0\;
            \While{loss(curr\_w + this\_eta * direction) $>$ loss(curr\_w) + beta * this\_eta * inner\_product(gradients, direction)}{
                this\_eta $\leftarrow$ this\_eta * gamma \tcp{Update eta using gamma, backtracking line search}
            }
            updates $\leftarrow$ this\_eta * direction\;
            escape\_saddle $\leftarrow$ False\;
        }
        \Else{
            buffer\_step $\leftarrow$ buffer\_step + 1\;
            \If{buffer\_step $==$ buffer\_limit}{
                escape\_saddle $\leftarrow$ True\;
                buffer\_step $\leftarrow$ 0\;
            }
            updates $\leftarrow$ -learning\_rate * gradients\;
        }
    }
    \Else{
        escape\_saddle $\leftarrow$ False\;
        updates $\leftarrow$ -learning\_rate * gradients\;
    }
    \Return updates, $\{'curr\_iter': \text{curr\_iter}, 't\_noise': \text{t\_noise}, 'curr\_w': \text{curr\_w} + \text{updates}\}$
}

\end{algorithm}

\begin{algorithm}[ht]\label{algo:tensor_PCA}
\DontPrintSemicolon
\caption{Tensor PCA Algorithm}

\textbf{Input:} tensor, lr, epochs, gradnorm\_epsilon, lambd\_v, key\\

\SetKwFunction{tensorPCA}{tensor\_PCA}
\SetKwFunction{loss}{loss}
\SetKwFunction{adamOptimize}{adam\_optimize}

\SetKwProg{Fn}{Function}{}{}

\Fn{\tensorPCA{tensor, lr, epochs, gradnorm\_epsilon, lambd\_v, key}}{
    \Fn{\loss{eigenval\_eigenvec, tensor}}{
        lambd, v $\leftarrow$ eigenval\_eigenvec\;
        k $\leftarrow$ len(tensor.shape)\;
        \For{each element in tensor.shape}{
            tensor $\leftarrow$ inner(tensor, v)\;
        }
        first\_term $\leftarrow$ square(lambd) * power(norm(v), 2*k)\;
        res $\leftarrow$ first\_term - 2*lambd*tensor\;
        \Return res\;
    }
    s $\leftarrow$ tensor.shape[0]\;
    \If{lambd\_v is None}{
        v $\leftarrow$ random.normal(shape=(s,)) / sqrt(s)\;
        lambd $\leftarrow$ 0.001 * random.normal()\;
    }
    \Else{
        lambd, v $\leftarrow$ lambd\_v\;
    }
    loss, grads, lambd\_v $\leftarrow$ \adamOptimize{(loss, (lambd, v), tensor), lr, epochs, gradnorm\_epsilon}\;
    lambd, v $\leftarrow$ lambd\_v\;
    sign $\leftarrow$ sign(lambd)\;
    
    \Return sign * power(abs(lambd), 1 / len(tensor.shape)) * v\;
}

\end{algorithm}
\end{document}